\documentclass{amsart}
\usepackage[latin1]{inputenc}
\usepackage[english]{babel}
\usepackage{csquotes}

\usepackage{color} 
\usepackage[colorlinks=true,linktoc=all,linkcolor=blue,citecolor=red,filecolor=pink,urlcolor=blue]{hyperref}
\usepackage[hyperref=true,style=alphabetic,citestyle=alphabetic,sorting=nyt,backend=bibtex,maxnames=100,doi=false,isbn=false,url=false,giveninits=true]{biblatex}
\bibliography{biblio}

\usepackage{graphicx}
\graphicspath{{pictures/}}

\usepackage[mathcal,mathscr]{euscript}
\usepackage{mathrsfs}

\usepackage{tikz}
\usetikzlibrary{arrows}

\usepackage[textsize=small]{todonotes}

\usepackage{fouriernc}
\usepackage[scaled=0.875]{helvet}
\usepackage{courier}

\usepackage{amssymb,amsmath,latexsym}
\theoremstyle{plain}                    
\newtheorem{theorem}{Theorem}[section]
\newtheorem{lemma}[theorem]{Lemma}
\newtheorem{proposition}[theorem]{Proposition}
\newtheorem{corollary}[theorem]{Corollary}
\newtheorem{question}[theorem]{Question}
\theoremstyle{definition}

\theoremstyle{remark}
\newtheorem{remark}[theorem]{Remark}

\numberwithin{equation}{section}

\newcommand{\zz}{\mathbb Z}

\newcommand{\rr}{\mathbb R}

\newcommand{\hh}{\mathbb H}

\newcommand{\pcd}{\operatorname{pcd}}

\newcommand{\six}{\mathsf{C}_{600}}

\newcommand{\Sb}{\mathbb{S}}

\newcommand{\PS}{\#}

\usepackage[inline]{enumitem}
\setenumerate[1]{label=\textbf{\arabic*.}}
\renewcommand{\leq}{\leqslant}
\renewcommand{\geq}{\geqslant}
\renewcommand{\setminus}{\smallsetminus}

  \usepackage[breakable]{tcolorbox}
    \usepackage{parskip} 

    \usepackage{caption}

    \usepackage{float}
    \usepackage{xcolor} 
    \usepackage{geometry} 
    \usepackage{textcomp} 
    \AtBeginDocument{%
    }
    \usepackage{upquote} 
    \usepackage{eurosym} 

    \usepackage{iftex}
    \ifPDFTeX
        \usepackage[T1]{fontenc}
        \IfFileExists{alphabeta.sty}{
              \usepackage{alphabeta}
          }{
              \usepackage[mathletters]{ucs}
              \usepackage[utf8x]{inputenc}
          }
    \else
        \usepackage{fontspec}
        \usepackage{unicode-math}
    \fi

    \usepackage{fancyvrb} 
    \usepackage[Export]{adjustbox} 
    \adjustboxset{max size={0.9\linewidth}{0.9\paperheight}}

    \usepackage{longtable} 
    \usepackage{booktabs}  
    \usepackage{array}     
    \usepackage{calc}      
    \usepackage[normalem]{ulem} 
 \usepackage{mathrsfs}

    \definecolor{urlcolor}{rgb}{0,.145,.698}
    \definecolor{linkcolor}{rgb}{.71,0.21,0.01}
    \definecolor{citecolor}{rgb}{.12,.54,.11}

    \definecolor{ansi-black}{HTML}{3E424D}
    \definecolor{ansi-black-intense}{HTML}{282C36}
    \definecolor{ansi-red}{HTML}{E75C58}
    \definecolor{ansi-red-intense}{HTML}{B22B31}
    \definecolor{ansi-green}{HTML}{00A250}
    \definecolor{ansi-green-intense}{HTML}{007427}
    \definecolor{ansi-yellow}{HTML}{DDB62B}
    \definecolor{ansi-yellow-intense}{HTML}{B27D12}
    \definecolor{ansi-blue}{HTML}{208FFB}
    \definecolor{ansi-blue-intense}{HTML}{0065CA}
    \definecolor{ansi-magenta}{HTML}{D160C4}
    \definecolor{ansi-magenta-intense}{HTML}{A03196}
    \definecolor{ansi-cyan}{HTML}{60C6C8}
    \definecolor{ansi-cyan-intense}{HTML}{258F8F}
    \definecolor{ansi-white}{HTML}{C5C1B4}
    \definecolor{ansi-white-intense}{HTML}{A1A6B2}
    \definecolor{ansi-default-inverse-fg}{HTML}{FFFFFF}
    \definecolor{ansi-default-inverse-bg}{HTML}{000000}

    \definecolor{outerrorbackground}{HTML}{FFDFDF}

    \DefineVerbatimEnvironment{Highlighting}{Verbatim}{commandchars=\\\{\}}

    \let\Oldtex\TeX
    \let\Oldlatex\LaTeX
    \renewcommand{\TeX}{\textrm{\Oldtex}}
    \renewcommand{\LaTeX}{\textrm{\Oldlatex}}

\makeatletter
\def\PY@reset{\let\PY@it=\relax \let\PY@bf=\relax%
    \let\PY@ul=\relax \let\PY@tc=\relax%
    \let\PY@bc=\relax \let\PY@ff=\relax}
\def\PY@tok#1{\csname PY@tok@#1\endcsname}
\def\PY@toks#1+{\ifx\relax#1\empty\else%
    \PY@tok{#1}\expandafter\PY@toks\fi}
\def\PY@do#1{\PY@bc{\PY@tc{\PY@ul{%
    \PY@it{\PY@bf{\PY@ff{#1}}}}}}}
\def\PY#1#2{\PY@reset\PY@toks#1+\relax+\PY@do{#2}}

\@namedef{PY@tok@w}{\def\PY@tc##1{\textcolor[rgb]{0.73,0.73,0.73}{##1}}}
\@namedef{PY@tok@c}{\let\PY@it=\textit\def\PY@tc##1{\textcolor[rgb]{0.24,0.48,0.48}{##1}}}
\@namedef{PY@tok@cp}{\def\PY@tc##1{\textcolor[rgb]{0.61,0.40,0.00}{##1}}}
\@namedef{PY@tok@k}{\let\PY@bf=\textbf\def\PY@tc##1{\textcolor[rgb]{0.00,0.50,0.00}{##1}}}
\@namedef{PY@tok@kp}{\def\PY@tc##1{\textcolor[rgb]{0.00,0.50,0.00}{##1}}}
\@namedef{PY@tok@kt}{\def\PY@tc##1{\textcolor[rgb]{0.69,0.00,0.25}{##1}}}
\@namedef{PY@tok@o}{\def\PY@tc##1{\textcolor[rgb]{0.40,0.40,0.40}{##1}}}
\@namedef{PY@tok@ow}{\let\PY@bf=\textbf\def\PY@tc##1{\textcolor[rgb]{0.67,0.13,1.00}{##1}}}
\@namedef{PY@tok@nb}{\def\PY@tc##1{\textcolor[rgb]{0.00,0.50,0.00}{##1}}}
\@namedef{PY@tok@nf}{\def\PY@tc##1{\textcolor[rgb]{0.00,0.00,1.00}{##1}}}
\@namedef{PY@tok@nc}{\let\PY@bf=\textbf\def\PY@tc##1{\textcolor[rgb]{0.00,0.00,1.00}{##1}}}
\@namedef{PY@tok@nn}{\let\PY@bf=\textbf\def\PY@tc##1{\textcolor[rgb]{0.00,0.00,1.00}{##1}}}
\@namedef{PY@tok@ne}{\let\PY@bf=\textbf\def\PY@tc##1{\textcolor[rgb]{0.80,0.25,0.22}{##1}}}
\@namedef{PY@tok@nv}{\def\PY@tc##1{\textcolor[rgb]{0.10,0.09,0.49}{##1}}}
\@namedef{PY@tok@no}{\def\PY@tc##1{\textcolor[rgb]{0.53,0.00,0.00}{##1}}}
\@namedef{PY@tok@nl}{\def\PY@tc##1{\textcolor[rgb]{0.46,0.46,0.00}{##1}}}
\@namedef{PY@tok@ni}{\let\PY@bf=\textbf\def\PY@tc##1{\textcolor[rgb]{0.44,0.44,0.44}{##1}}}
\@namedef{PY@tok@na}{\def\PY@tc##1{\textcolor[rgb]{0.41,0.47,0.13}{##1}}}
\@namedef{PY@tok@nt}{\let\PY@bf=\textbf\def\PY@tc##1{\textcolor[rgb]{0.00,0.50,0.00}{##1}}}
\@namedef{PY@tok@nd}{\def\PY@tc##1{\textcolor[rgb]{0.67,0.13,1.00}{##1}}}
\@namedef{PY@tok@s}{\def\PY@tc##1{\textcolor[rgb]{0.73,0.13,0.13}{##1}}}
\@namedef{PY@tok@sd}{\let\PY@it=\textit\def\PY@tc##1{\textcolor[rgb]{0.73,0.13,0.13}{##1}}}
\@namedef{PY@tok@si}{\let\PY@bf=\textbf\def\PY@tc##1{\textcolor[rgb]{0.64,0.35,0.47}{##1}}}
\@namedef{PY@tok@se}{\let\PY@bf=\textbf\def\PY@tc##1{\textcolor[rgb]{0.67,0.36,0.12}{##1}}}
\@namedef{PY@tok@sr}{\def\PY@tc##1{\textcolor[rgb]{0.64,0.35,0.47}{##1}}}
\@namedef{PY@tok@ss}{\def\PY@tc##1{\textcolor[rgb]{0.10,0.09,0.49}{##1}}}
\@namedef{PY@tok@sx}{\def\PY@tc##1{\textcolor[rgb]{0.00,0.50,0.00}{##1}}}
\@namedef{PY@tok@m}{\def\PY@tc##1{\textcolor[rgb]{0.40,0.40,0.40}{##1}}}
\@namedef{PY@tok@gh}{\let\PY@bf=\textbf\def\PY@tc##1{\textcolor[rgb]{0.00,0.00,0.50}{##1}}}
\@namedef{PY@tok@gu}{\let\PY@bf=\textbf\def\PY@tc##1{\textcolor[rgb]{0.50,0.00,0.50}{##1}}}
\@namedef{PY@tok@gd}{\def\PY@tc##1{\textcolor[rgb]{0.63,0.00,0.00}{##1}}}
\@namedef{PY@tok@gi}{\def\PY@tc##1{\textcolor[rgb]{0.00,0.52,0.00}{##1}}}
\@namedef{PY@tok@gr}{\def\PY@tc##1{\textcolor[rgb]{0.89,0.00,0.00}{##1}}}
\@namedef{PY@tok@ge}{\let\PY@it=\textit}
\@namedef{PY@tok@gs}{\let\PY@bf=\textbf}
\@namedef{PY@tok@ges}{\let\PY@bf=\textbf\let\PY@it=\textit}
\@namedef{PY@tok@gp}{\let\PY@bf=\textbf\def\PY@tc##1{\textcolor[rgb]{0.00,0.00,0.50}{##1}}}
\@namedef{PY@tok@go}{\def\PY@tc##1{\textcolor[rgb]{0.44,0.44,0.44}{##1}}}
\@namedef{PY@tok@gt}{\def\PY@tc##1{\textcolor[rgb]{0.00,0.27,0.87}{##1}}}
\@namedef{PY@tok@err}{\def\PY@bc##1{{\setlength{\fboxsep}{\string -\fboxrule}\fcolorbox[rgb]{1.00,0.00,0.00}{1,1,1}{\strut ##1}}}}
\@namedef{PY@tok@kc}{\let\PY@bf=\textbf\def\PY@tc##1{\textcolor[rgb]{0.00,0.50,0.00}{##1}}}
\@namedef{PY@tok@kd}{\let\PY@bf=\textbf\def\PY@tc##1{\textcolor[rgb]{0.00,0.50,0.00}{##1}}}
\@namedef{PY@tok@kn}{\let\PY@bf=\textbf\def\PY@tc##1{\textcolor[rgb]{0.00,0.50,0.00}{##1}}}
\@namedef{PY@tok@kr}{\let\PY@bf=\textbf\def\PY@tc##1{\textcolor[rgb]{0.00,0.50,0.00}{##1}}}
\@namedef{PY@tok@bp}{\def\PY@tc##1{\textcolor[rgb]{0.00,0.50,0.00}{##1}}}
\@namedef{PY@tok@fm}{\def\PY@tc##1{\textcolor[rgb]{0.00,0.00,1.00}{##1}}}
\@namedef{PY@tok@vc}{\def\PY@tc##1{\textcolor[rgb]{0.10,0.09,0.49}{##1}}}
\@namedef{PY@tok@vg}{\def\PY@tc##1{\textcolor[rgb]{0.10,0.09,0.49}{##1}}}
\@namedef{PY@tok@vi}{\def\PY@tc##1{\textcolor[rgb]{0.10,0.09,0.49}{##1}}}
\@namedef{PY@tok@vm}{\def\PY@tc##1{\textcolor[rgb]{0.10,0.09,0.49}{##1}}}
\@namedef{PY@tok@sa}{\def\PY@tc##1{\textcolor[rgb]{0.73,0.13,0.13}{##1}}}
\@namedef{PY@tok@sb}{\def\PY@tc##1{\textcolor[rgb]{0.73,0.13,0.13}{##1}}}
\@namedef{PY@tok@sc}{\def\PY@tc##1{\textcolor[rgb]{0.73,0.13,0.13}{##1}}}
\@namedef{PY@tok@dl}{\def\PY@tc##1{\textcolor[rgb]{0.73,0.13,0.13}{##1}}}
\@namedef{PY@tok@s2}{\def\PY@tc##1{\textcolor[rgb]{0.73,0.13,0.13}{##1}}}
\@namedef{PY@tok@sh}{\def\PY@tc##1{\textcolor[rgb]{0.73,0.13,0.13}{##1}}}
\@namedef{PY@tok@s1}{\def\PY@tc##1{\textcolor[rgb]{0.73,0.13,0.13}{##1}}}
\@namedef{PY@tok@mb}{\def\PY@tc##1{\textcolor[rgb]{0.40,0.40,0.40}{##1}}}
\@namedef{PY@tok@mf}{\def\PY@tc##1{\textcolor[rgb]{0.40,0.40,0.40}{##1}}}
\@namedef{PY@tok@mh}{\def\PY@tc##1{\textcolor[rgb]{0.40,0.40,0.40}{##1}}}
\@namedef{PY@tok@mi}{\def\PY@tc##1{\textcolor[rgb]{0.40,0.40,0.40}{##1}}}
\@namedef{PY@tok@il}{\def\PY@tc##1{\textcolor[rgb]{0.40,0.40,0.40}{##1}}}
\@namedef{PY@tok@mo}{\def\PY@tc##1{\textcolor[rgb]{0.40,0.40,0.40}{##1}}}
\@namedef{PY@tok@ch}{\let\PY@it=\textit\def\PY@tc##1{\textcolor[rgb]{0.24,0.48,0.48}{##1}}}
\@namedef{PY@tok@cm}{\let\PY@it=\textit\def\PY@tc##1{\textcolor[rgb]{0.24,0.48,0.48}{##1}}}
\@namedef{PY@tok@cpf}{\let\PY@it=\textit\def\PY@tc##1{\textcolor[rgb]{0.24,0.48,0.48}{##1}}}
\@namedef{PY@tok@c1}{\let\PY@it=\textit\def\PY@tc##1{\textcolor[rgb]{0.24,0.48,0.48}{##1}}}
\@namedef{PY@tok@cs}{\let\PY@it=\textit\def\PY@tc##1{\textcolor[rgb]{0.24,0.48,0.48}{##1}}}

\makeatother

    \makeatletter
        \newbox\Wrappedcontinuationbox 
        \newbox\Wrappedvisiblespacebox 
        \newcommand*\Wrappedvisiblespace {\textcolor{red}{\textvisiblespace}} 
        \newcommand*\Wrappedcontinuationsymbol {\textcolor{red}{\llap{\tiny$\m@th\hookrightarrow$}}} 
        \newcommand*\Wrappedcontinuationindent {3ex } 
        \newcommand*\Wrappedafterbreak {\kern\Wrappedcontinuationindent\copy\Wrappedcontinuationbox} 
        \newcommand*\Wrappedbreaksatspecials {%
            \def\PYGZus{\discretionary{\char`\_}{\Wrappedafterbreak}{\char`\_}}%
            \def\PYGZob{\discretionary{}{\Wrappedafterbreak\char`\{}{\char`\{}}%
            \def\PYGZcb{\discretionary{\char`\}}{\Wrappedafterbreak}{\char`\}}}%
            \def\PYGZca{\discretionary{\char`\^}{\Wrappedafterbreak}{\char`\^}}%
            \def\PYGZam{\discretionary{\char`\&}{\Wrappedafterbreak}{\char`\&}}%
            \def\PYGZlt{\discretionary{}{\Wrappedafterbreak\char`\<}{\char`\<}}%
            \def\PYGZgt{\discretionary{\char`\>}{\Wrappedafterbreak}{\char`\>}}%
            \def\PYGZsh{\discretionary{}{\Wrappedafterbreak\char`\#}{\char`\#}}%
            \def\PYGZpc{\discretionary{}{\Wrappedafterbreak\char`\%}{\char`\%}}%
            \def\PYGZdl{\discretionary{}{\Wrappedafterbreak\char`\$}{\char`\$}}%
            \def\PYGZhy{\discretionary{\char`\-}{\Wrappedafterbreak}{\char`\-}}%
            \def\PYGZsq{\discretionary{}{\Wrappedafterbreak\textquotesingle}{\textquotesingle}}%
            \def\PYGZdq{\discretionary{}{\Wrappedafterbreak\char`\"}{\char`\"}}%
            \def\PYGZti{\discretionary{\char`\~}{\Wrappedafterbreak}{\char`\~}}%
        } 
        \newcommand*\Wrappedbreaksatpunct {%
            \lccode`\~`\.\lowercase{\def~}{\discretionary{\hbox{\char`\.}}{\Wrappedafterbreak}{\hbox{\char`\.}}}%
            \lccode`\~`\,\lowercase{\def~}{\discretionary{\hbox{\char`\,}}{\Wrappedafterbreak}{\hbox{\char`\,}}}%
            \lccode`\~`\;\lowercase{\def~}{\discretionary{\hbox{\char`\;}}{\Wrappedafterbreak}{\hbox{\char`\;}}}%
            \lccode`\~`\:\lowercase{\def~}{\discretionary{\hbox{\char`\:}}{\Wrappedafterbreak}{\hbox{\char`\:}}}%
            \lccode`\~`\?\lowercase{\def~}{\discretionary{\hbox{\char`\?}}{\Wrappedafterbreak}{\hbox{\char`\?}}}%
            \lccode`\~`\!\lowercase{\def~}{\discretionary{\hbox{\char`\!}}{\Wrappedafterbreak}{\hbox{\char`\!}}}%
            \lccode`\~`\/\lowercase{\def~}{\discretionary{\hbox{\char`\/}}{\Wrappedafterbreak}{\hbox{\char`\/}}}%
            \catcode`\.\active
            \catcode`\,\active 
            \catcode`\;\active
            \catcode`\:\active
            \catcode`\?\active
            \catcode`\!\active
            \catcode`\/\active 
            \lccode`\~`\~ 	
        }
    \makeatother

    \let\OriginalVerbatim=\Verbatim
    \makeatletter
    \renewcommand{\Verbatim}[1][1]{%
        \sbox\Wrappedcontinuationbox {\Wrappedcontinuationsymbol}%
        \sbox\Wrappedvisiblespacebox {\FV@SetupFont\Wrappedvisiblespace}%
        \def\FancyVerbFormatLine ##1{\hsize\linewidth
            \vtop{\raggedright\hyphenpenalty\z@\exhyphenpenalty\z@
                \doublehyphendemerits\z@\finalhyphendemerits\z@
                \strut ##1\strut}%
        }%
        \def\FV@Space {%
            \nobreak\hskip\z@ plus\fontdimen3\font minus\fontdimen4\font
            \discretionary{\copy\Wrappedvisiblespacebox}{\Wrappedafterbreak}
            {\kern\fontdimen2\font}%
        }%
        
        \Wrappedbreaksatspecials
        \OriginalVerbatim[#1,codes*=\Wrappedbreaksatpunct]%
    }
    \makeatother

    \definecolor{incolor}{HTML}{303F9F}
    \definecolor{outcolor}{HTML}{D84315}
    \definecolor{cellborder}{HTML}{CFCFCF}
    \definecolor{cellbackground}{HTML}{F7F7F7}
    
    \makeatletter
    \newcommand{\boxspacing}{\kern\kvtcb@left@rule\kern\kvtcb@boxsep}
    \makeatother

\makeatletter
\@namedef{subjclassname@2020}{\textup{2020} Mathematics Subject Classification}
\makeatother

\title[Convex cocompact groups with three-dimensional limit sets]{Convex cocompact groups with three-dimensional limit sets}
\author{Sami Douba}
\address{Mathematisches Institut der Universit\"at Bonn, Endenicher Allee 60, 53115 Bonn, Germany}
\email{douba@math.uni-bonn.de}

\author{Gye-Seon Lee}
\address{Department of Mathematical Sciences and Research Institute of Mathematics, Seoul National University, Seoul 08826, South Korea}
\email{gyeseonlee@snu.ac.kr}

\author{Ludovic Marquis}
\address{Univ Rennes, CNRS, IRMAR - UMR 6625, F-35000 Rennes, France}
\email{ludovic.marquis@univ-rennes.fr}

\author{Lorenzo Ruffoni}
\address{Department of Mathematics and Statistics - Binghamton University, Binghamton, NY 13902, USA}
\email{lorenzo.ruffoni2@gmail.com}

\begin{document}

\date{\today}
\subjclass[2020]{20F65, 20F67, 22E40 (primary); 20F55 (secondary)}

\begin{abstract}
We provide a general construction of convex cocompact hyperbolic reflection groups with three-dimensional limit sets. More precisely, our construction takes as input an arbitrary simplicial complex $L$ of dimension $3$ on $n$ vertices, and outputs a convex cocompact right-angled reflection group acting on real hyperbolic $n$-space whose nerve is precisely the Przytycki--\'{S}wi\k{a}tkowski subdivision of $L$. Moreover, the output reflection group is a thin subgroup of an $n$-dimensional cocompact arithmetic hyperbolic lattice. This answers affirmatively a question of M.~Kapovich concerning the existence of a convex cocompact group acting on some real hyperbolic space with limit set a \v Cech cohomology sphere other than the standard sphere. 
\end{abstract}

\keywords{reflection groups, convex cocompact subgroups, real hyperbolic spaces, limit sets, right-angled Coxeter groups, Pontryagin surfaces, trees of manifolds, Jakobsche spaces}

\maketitle

\section{Introduction}

The Gromov boundary of a Gromov-hyperbolic group $\Gamma$ is a well-defined topological invariant of $\Gamma$ that is compact and metrizable \cite{G87}. It is natural to ask which compact metrizable spaces arise as the boundary of some Gromov-hyperbolic group. While the possible boundaries of topological dimension one are well understood \cite{KK00}, this problem remains largely open in higher dimensions. Indeed, beyond spheres, relatively few examples of dimension $>1$ are known; notable exceptions include certain Menger and Sierpi\'nski compacta, Pontryagin surfaces, and many trees of manifolds \cite{DR99,FI03,DO07,PS09,SW20}.

A closely related question arises in the study of discrete subgroups of the isometry group $\mathrm{Isom}(\mathbb{H}^n)$ of real hyperbolic $n$-space $\mathbb{H}^n$. The action of such a subgroup $\Gamma$ on $\mathbb{H}^n$ extends to the visual boundary $\partial_\infty \mathbb{H}^n$ of $\mathbb{H}^n$, namely, the conformal sphere $\Sb^{n-1}$, and in the case that $\Gamma$ is not virtually abelian, there is a unique closed nonempty minimal $\Gamma$-invariant subset $\Lambda_\Gamma \subset \partial_\infty \mathbb{H}^n$, called the {\em limit set} of $\Gamma$. The limit set is a fundamental object of study, as it encodes key geometric and dynamical information about the group $\Gamma$. In particular, if $\Gamma$ is convex cocompact, then $\Gamma$ is Gromov-hyperbolic, and the Gromov boundary of $\Gamma$ is equivariantly homeomorphic to $\Lambda_\Gamma$ \cite{KAP09,SU79}. Despite considerable interest, the possible topological types of limit sets remain poorly understood; for work in this direction, see, for instance, \cite{BO97, DLMR26,Ma25,MZ25}.

In the present work, we address the following question of M. Kapovich.

\begin{question}{\rm \cite[Question 9.4]{Kapovich08}} \label{question:kapovich}\,
Is there a convex cocompact subgroup of $\mathrm{Isom}(\mathbb{H}^n)$ for some dimension $n$ whose limit set is a \v{C}ech cohomology sphere but is not homeomorphic to a sphere?
\end{question}

We remark that, if $\Gamma < \mathrm{Isom}(\hh^n)$ is convex cocompact, then $\Lambda_\Gamma$ is a \v{C}ech cohomology sphere if and only if $\Gamma$ is a virtual Poincar\'e duality group \cite{BM91}. In this paper, we resolve Question~\ref{question:kapovich} affirmatively by providing a general construction of convex cocompact subgroups of $\mathrm{Isom}(\mathbb{H}^n)$ with $2$- and $3$-dimensional limit sets.

It is known that an abstract right-angled Coxeter group is Gromov-hyperbolic if and only if its nerve is so-called {\em flag-no-square}; see \S\ref{subsection:Coxeter_group} for definitions. Dranishnikov \cite{DR97} introduced a procedure that takes as input an arbitrary simplicial $2$-complex $L$ and outputs a flag-no-square subdivision $L^\PS$ of $L$. This procedure was generalized to dimension $\leqslant 3$ by Przytycki and \'{S}wi\k{a}tkowski \cite{PS09}; we will continue to denote the Przytycki--\'{S}wi\k{a}tkowski subdivision of a simplicial $3$-complex $L$ by $L^\PS$. Our main result is as follows.

\begin{theorem}\label{thm:main}
 For each $n \geqslant 4$, there is a cocompact arithmetic lattice  $\Delta_n < \mathrm{Isom}(\hh^n)$ such that, for any simplicial complex $L$ on $n$ vertices and of dimension $d \leqslant 3$, the lattice $\Delta_n$ contains a right-angled reflection subgroup $\Gamma_L$ whose nerve is $L^\PS$. Moreover, if $d \geqslant 1$, then $\Gamma_L$ can be chosen to be Zariski-dense\footnote{It is not difficult to see that the condition $d \geq 1$ is also necessary here.} in $\mathrm{Isom}(\hh^n)$.
\end{theorem}

We remark that, for $n\leq 3$, we have that $L$ is a subcomplex of the $2$-simplex, so that $L^\PS$ is a full subcomplex of the icosahedron. Thus, in this case, there is a reflection subgroup with nerve $L^\PS$ inside the reflection group associated to the right-angled dodecahedron in $\hh^3$.

Note that a finitely generated reflection group $\Gamma < \mathrm{Isom}(\hh^n)$ is geometrically finite, so that if $\Gamma$ is moreover contained in a cocompact lattice in $\mathrm{Isom}(\hh^n)$, then $\Gamma$ is automatically convex cocompact. Nevertheless, it will be clear from our argument that the reflection groups we construct are convex cocompact, using the fact that a group generated by the reflections in the walls of a finite-sided right-angled hyperbolic polyhedron $P$ is convex cocompact if and only if no two walls of $P$ are asymptotic (see \cite[Theorem~4.7]{DH13}). In particular, we recover Przytycki--\'{S}wi\k{a}tkowski's result~\cite[Proposition~2.13]{PS09} that their subdivision $L^\PS$ is flag-no-square.

\begin{remark}\label{rem:thin}
    Any subgroup $\Gamma_L < \mathrm{Isom}(\hh^n)$ as in Theorem~\ref{thm:main} is necessarily of infinite covolume, so that the subgroups $\Gamma_L < \Delta_n$ we construct are \textit{thin} in the sense of Sarnak (see \cite{BO14}). Indeed, if such a subgroup $\Gamma_L$ were to be cocompact in $\mathrm{Isom}(\hh^n)$, then, since the boundary of a compact convex polyhedron in $\hh^n$ is topologically an $(n-1)$-sphere $S^{n-1}$, the complex $L$ would have to be an $n$-vertex triangulation of $S^{n-1}$, and no such triangulation exists.
\end{remark}

The following are applications of Theorem~\ref{thm:main}.

\begin{corollary}\label{cor:Jakobsche_space}
Let $d\leq 3$ and let $N$ be a closed connected  $d$-manifold admitting a triangulation with $n$ vertices. 
Then there is a convex cocompact reflection group in $\mathrm{Isom}(\mathbb{H}^{n})$ whose limit set is homeomorphic to the tree of manifolds $\mathcal X(N)$ (if $N$ is nonorientable) or $\mathcal X(N \# \overline N)$ (if $N$ is orientable).
\end{corollary}

See \S\ref{subsection:trees} for details on trees of manifolds (also known as {\em Jakobsche spaces}). In particular, using triangulations of the sphere $S^d$ with $n$ vertices, one obtains thin convex cocompact reflection groups acting on $\hh^n$ with limit set~$S^d$.
On the other hand, using triangulations of  the real projective plane, the torus,  and the Poincar\'e homology sphere, one obtains the following limit sets.
For Item~(\ref{item:h16}) in Corollary~\ref{cor:manifold}, we use a triangulation of the Poincar\'e homology $3$-sphere with $16$ vertices due to Bj\"{o}rner and Lutz \cite{BL00}, which is the triangulation with the smallest number of vertices known to the authors.

\begin{corollary}\label{cor:manifold}
    There are convex cocompact reflection groups
    \begin{enumerate}
        \item\label{item:h6} in $\mathrm{Isom} (\hh^6)$ whose limit set is the Pontryagin surface $\Pi_2$;

        \item\label{item:h7} in $\mathrm{Isom} (\hh^7)$ whose limit set is the orientable Pontryagin sphere;

        \item\label{item:h16} \label{item:kapovich} in $\mathrm{Isom} (\hh^{16})$ whose limit set is a \v{C}ech cohomology 3-sphere not homeomorphic to $S^3$.
        
    \end{enumerate}
\end{corollary}

Ma--Zheng \cite{MZ25} previously constructed   convex cocompact reflection groups in $\mathrm{Isom} (\hh^5)$ whose limit sets are the Pontryagin surfaces $\Pi_2$ and $\Pi_3$.
In \cite{DLMR26}, the authors exhibited a convex cocompact reflection group in $\mathrm{Isom} (\hh^4)$ whose limit set is the orientable Pontryagin sphere.

\begin{remark}
From Corollary~\ref{cor:Jakobsche_space}, one in fact obtains infinitely many pairwise non-homeomorphic \v{C}ech cohomology $3$-spheres as limit sets of convex cocompact groups.
This can for instance be achieved by letting $N$ vary among the Brieskorn homology $3$-spheres; see \cite[\S11]{JA91} and \cite{MI75}.
\end{remark}

For the next result, we use triangulations of the $p$-fold dunce hat; see \S\ref{sec:dunce} and Figure~\ref{fig:dunce_hat}.

\begin{corollary}\label{cor:pontryagin}
 For each prime $p \geq 2$, there is a convex cocompact reflection group in $\mathrm{Isom}(\mathbb{H}^{n})$ with
$n = \big\lceil \tfrac{3p}{2} \big\rceil + 3$ whose limit set is a Pontryagin surface $\Pi_p$.
\end{corollary}

We also show that many of the groups obtained in Corollary~\ref{cor:Jakobsche_space} admit convex cocompact reflection subgroups with limit set a Menger curve; see \cite[Proposition 5.1]{DLMR26} and Proposition~\ref{prop:menger}.

We remark that the dimension $n$ in the construction of Theorem~\ref{thm:main} has no reason to be optimal. 
For example, for $d \leq 3$, the minimal number of vertices in a triangulation of the $d$-sphere is $d+2$, and so Theorem~\ref{thm:main} produces a subgroup of $\mathrm{Isom} (\hh^{d+2})$ whose limit set is a $d$-sphere within $\partial_\infty \mathrm{\hh^{d+2}}$. However, any cocompact lattice in $\mathrm{Isom}(\hh^{d+1})$, which in these dimensions may be taken to be a right-angled reflection group, has limit set $S^d$.

Similarly, we expect that there is a convex cocompact subgroup of $\mathrm{Isom}(\mathbb{H}^{d})$ for some integer $d < 16$ whose limit set is a \v{C}ech cohomology sphere not homeomorphic to the standard sphere. 
Indeed, each homology 3-sphere admits a locally flat embedding in $S^4$ by a result of Freedman \cite[Theorem 1.4]{FR82}.
Hence, using \cite[Example 9.1]{JA91}, one obtains that the \v{C}ech cohomology 3-sphere given by Corollary~\ref{cor:manifold} embeds in $S^4$. 
A natural question for future research is the following.

\begin{question}
Let $X\subseteq \Sb^d = \partial_\infty \hh^{d+1}$ and suppose there is a convex cocompact subgroup of $\mathrm{Isom}(\mathbb{H}^{n})$ for some integer $n \geq 4$ with limit set homeomorphic to $X$.
Is there a convex cocompact subgroup of $\mathrm{Isom}(\mathbb{H}^{d+1})$  with limit set homeomorphic to $X$?
\end{question}

Note that if $\Gamma$ is a convex cocompact subgroup of $\mathrm{Isom}(\mathbb{H}^{3})$ that does not split over virtually cyclic subgroups, then 
the topological type of $\Lambda_\Gamma$
is well understood. Namely, the limit set $\Lambda_\Gamma$ must either be a circle, a Sierpi\'nski carpet, or $S^2$ itself (see \cite{KK00}). For a survey on existing knowledge regarding which Gromov-hyperbolic (Coxeter) groups can be realized as convex cocompact (reflection) subgroups of $\mathrm{Isom}(\hh^n)$, we refer the reader to~\cite[\S4.2.8]{DA24}.

\subsection*{Acknowledgements} 
We thank Mike Davis for his interest in this work. We are also grateful to Jiming Ma for sharing his preprints on recent joint work with Fangting Zheng concerning convex cocompact subgroups with Pontryagin surface limit sets $\Pi_2$ and $\Pi_3$, as well as his independent work on $2$-dimensional Menger compactum limit sets.

This material is based upon work supported by the National Science Foundation under Grant No.~DMS-2424139, while S.D. and G.L. were in residence at the Simons Laufer Mathematical Sciences Institute in Berkeley, California, during the Spring 2026 semester. 
G.L. was supported by the National Research Foundation of Korea(NRF) grant funded by the Korea government(MSIT) (No.~RS-2026-25491130). L.M. acknowledges support by the Centre Henri Lebesgue (ANR-11-LABX-0020 LEBESGUE),  ANR G\'eom\'etries de Hilbert sur tout corps valu\'e (ANR-23-CE40-0012) and  ANR Groupes Op\'erant sur des FRactales (ANR-22-CE40-0004).
L.R. acknowledges support by INDAM-GNSAGA.


\section{Preliminaries}
\subsection{Simplicial complexes and the subdivisions of Dranishnikov and Przytycki--\'{S}wi\k{a}tkowski}\label{sec:simplicial subs}
Let $L$ be a finite simplicial complex. 
We will use the following terminology and notation:
\begin{itemize}
    \item The $d$-\textit{skeleton} of $L$ is denoted by  $L^{(d)}$.

    \item If $K\subseteq L$ is a subcomplex, we say that $K$ is \textit{full} (or \textit{induced}) if  whenever $d+1$ vertices of $K$ span a $d$-simplex in $L$, they also span a $d$-simplex in $K$.

    \item $L$ is \textit{flag} if any collection of $d+1$ pairwise adjacent vertices spans a $d$-simplex in $L$. The complex $L$ is \textit{flag-no-square} if $L$ is flag and has no induced squares (i.e., every square has a chord). A full subcomplex of a flag (respectively, flag-no-square) complex is flag (respectively, flag-no-square).
\end{itemize}

Dranishnikov \cite{DR99} introduced a subdivision procedure to turn every 2-dimensional simplicial complex into a 2-dimensional flag-no-square simplicial complex.
This procedure consists of
adding a midpoint to every edge and then subdividing each triangle as shown in Figure~\ref{fig:dran}. Note that the complex on the right-hand side of the picture can be obtained by fixing a triangle $\tau$ in the standard icosahedron $I$ and then taking the full subcomplex of  $I$ spanned by all the vertices not contained in $\tau$.
If $L$ is 2-dimensional, then we denote by $L^\PS$ the flag-no-square complex obtained by applying this subdivision procedure.
\begin{figure}[ht]
\centering
\def\svgwidth{.5\columnwidth}
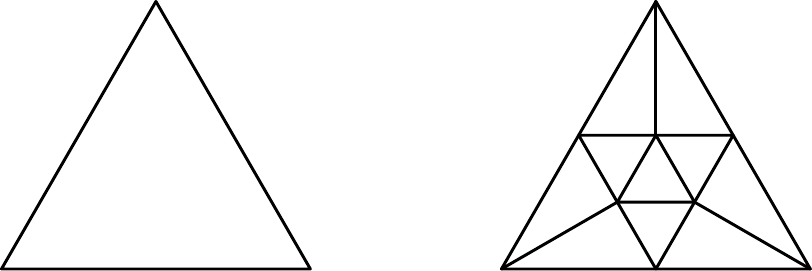  
\caption{Dranishnikov's subdivision procedure for 2-dimensional simplicial complexes.} 
\label{fig:dran}
\end{figure}

Dranishnikov's subdivision was generalized by Przytycki--\'{S}wi\k{a}tkowski  \cite{PS09} for simplicial complexes of dimension $\leqslant 3$.
In the latter procedure,  one applies Dranishnikov's subdivision to the $2$-skeleton and then
subdivides each tetrahedron into a copy of the $3$-dimensional complex obtained  as follows: Fix a tetrahedron $\tau$ in the 600-cell $\six$ and then take the full subcomplex of $\six$ spanned by all the vertices not contained in $\tau$.
If $L$ is 3-dimensional, then we continue to denote the Przytycki--\'{S}wi\k{a}tkowski subdivision of $L$ by $L^\PS$.

Note that both the icosahedron $I$ and the 600-cell $\six$ are dual to compact right-angled polyhedra in real hyperbolic space, namely, the dodecahedron and the 120-cell respectively (note such polyhedra cease to exist in dimensions higher than $4$; see \cite{PV05}). 
The Gromov hyperbolicity of the associated reflection groups is manifested in the fact that $I$ and $\six$ are flag-no-square; see \S\ref{sec:racg}. In the language of right-angled Coxeter groups (RACGs) introduced below, one can think of these subdivision procedures as a way to ``hyperbolize'' given RACGs using classical real hyperbolic right-angled reflection groups to produce  (abstract) Gromov-hyperbolic RACGs. Our goal will be to realize the output groups geometrically as concrete reflection groups acting on $\hh^n$.


\subsection{Right-angled Coxeter groups}\label{subsection:Coxeter_group}\label{sec:racg}
Given a flag simplicial complex $L$, the \textit{right-angled Coxeter group} (RACG) defined by $L$ is the group $W_L$ generated by one involution for each vertex of $L$, with two generators commuting precisely when the corresponding vertices are adjacent. In other words, the group $W_L$ is given by the presentation
$$
W_L=\langle s\in L^{(0)} \mid s^2=1, \; [s,t]=1 \textrm{ for all edges } (s,t) \textrm{ of }  L \rangle.
$$
The reader is referred to  \cite{DA08} or \cite{DA24} for background.
The flag complex $L$ is called the \textit{nerve} of the group $W_L$ (and depends only on the isomorphism type of $W_L$; see \cite{RA03}). 
Note that $W_L$ is completely determined by the $1$-skeleton of $L$. 
It is well-known that $W_L$ is Gromov-hyperbolic if and only if $L$ is flag-no-square; see \cite{moussong,DGK18}.
Moreover, for any full subcomplex $K\subseteq L$, 
the subgroups generated by $\langle K^{(0)}\rangle$ is naturally isomorphic to $W_K$ and is quasiconvex (with respect to the generating set $L^{(0)}$ for $W_L$). We call $W_K$ the \textit{standard subgroup} defined by the full subcomplex $K$.

The following argument is well known to experts; see \cite[\S9]{DA01} or \cite[Example~4.35]{DA24}. We sketch it for the reader's convenience.

 \begin{lemma}\label{lem:RACG homology 3-sphere}
     Let $L$ be a flag-no-square triangulation of an (integral) homology $3$-sphere $N$. 
     Then the Gromov boundary of $W_L$ is a \v Cech cohomology 3-sphere $S$. Moreover, we have that $S \cong S^3$ if and only if $N \cong S^3$.
 \end{lemma}
 \begin{proof}
     The Davis complex $X_L$ of $W_L$ is a 4-dimensional CAT(0) cube complex. The link of each vertex of $X_L$ is a copy of $N$.
     The commutator subgroup  $W_L'$ of $W_L$ acts freely and cocompactly on the Davis complex, and the quotient $P$ is a closed 4-pseudomanifold. Each vertex of $P$ has a neighborhood homeomorphic to the cone over~$N$.
     Since $N$ is a 3-manifold, the complement in $P$ of the vertex set of $P$ is a 4-manifold.
     
     By Freedman's work,  the homology $3$-sphere $N$ bounds a compact contractible 4-manifold $B$; see \cite[Theorem 1.4']{FR82}.
     If we replace a small neighborhood of every vertex of $P$ with a copy of $B$, we obtain a closed 4-manifold  that is homotopy equivalent to $P$, so that $W_L'$ is a $4$-dimensional Poincar\'e duality group. It then follows from work of Bestvina--Mess \cite[Corollary 1.3(c)]{BM91} that the Gromov boundary $S$ of $W_L'$ is a \v Cech cohomology 3-sphere.
     Since $W_L'$ has finite index in $W_L$, they share the same Gromov boundary $S$.
    Moreover, 
    if $N \not \cong S^3$, then $S \not \cong S^3$, as proved by Davis \cite[Proposition 9.4]{DA01}.
 \end{proof}

The RACGs considered in this paper arise  concretely as {\em right-angled hyperbolic reflection groups}, that is, as discrete groups generated by reflections in the codimension-$1$ faces of certain right-angled polyhedra in real hyperbolic $n$-space $\hh^n$. 
The \textit{limit set} $\Lambda_\Gamma$ of a discrete subgroup $\Gamma < \mathrm{Isom}(\hh^n)$ is the closed subset of $\Sb^{n-1} = \partial_\infty \hh^n$ consisting of all accumulation points of the $\Gamma$-orbit of some (any) point of $\hh^n$.
We say that $\Gamma$ is \textit{convex cocompact} if $\Gamma$ acts cocompactly on some closed convex $\Gamma$-invariant subset of $\hh^n$.
If $\Gamma$ is convex cocompact, then $\Gamma$ is Gromov-hyperbolic as an abstract group, and its limit set $\Lambda_\Gamma$ is equivariantly homeomorphic to the Gromov boundary of $\Gamma$. In particular, the topological type of the limit set of a convex cocompact right-angled hyperbolic reflection group $\Gamma$ is entirely determined by the nerve of $\Gamma$ as an abstract RACG.


\subsection{Some low-dimensional continua}\label{sec:continua}
The purpose of this section is to define the continua that appear as Gromov boundaries of certain Gromov-hyperbolic RACGs whose nerve is not a sphere.
(That the Coxeter groups in this discussion be  Gromov-hyperbolic or right-angled is not relevant here, and one has identical statements for arbitrary finitely generated Coxeter groups and the visual boundaries of their Davis complexes. In this paper, we are however only interested in the hyperbolic and right-angled case.)

\subsubsection{Trees of manifolds}\label{subsection:trees} 
Trees of manifolds were introduced by Jakobsche \cite{JA91}; see also \cite{SW20,SW20sp}, whose notation we follow.
 Roughly speaking, given a closed connected manifold $N$, the tree of manifolds $\mathcal X(N)$ is obtained as an inverse limit of connected sums of copies of $N$.
 When $N=S^d$ is a sphere, one has $\mathcal X(N)\cong S^d$, 
 but when $N\not\cong S^d$, the  tree of manifolds $\mathcal X(N)$ fails to be a manifold.
 For instance, if $N$ is an orientable surface of positive genus, then~$\mathcal X(N\# \overline N)$ is the so-called Pontryagin sphere, and if $N$ is a nonorientable surface, then~$\mathcal X(N)$ is the nonorientable Pontryagin sphere.

Trees of manifolds arise as boundaries of Coxeter groups with manifold nerve as follows (see   \cite{FI03,SW20}).
If $L$ is a flag-no-square triangulation of a closed connected manifold $N$, then  the Davis complex $X_L$ of $W_L$ is a pseudomanifold with isolated vertex singularities, where the link of each vertex is homeomorphic to $L$. 
The idea is then to realize the Gromov boundary of $W_L$ as the visual boundary of $X_L$, which can be computed as the inverse limit of an exhausting sequence of metric spheres centered at a point $o \in X_L$.
Informally, as the metric sphere centered at $o$ expands, each encounter with a vertex of $X_L$ contributes a connected sum with the link (a copy of $N$) of that vertex.

\begin{lemma}[Theorem 3.7 in \cite{FI03}, Theorem 2 in \cite{SW20}]\label{lem:tree}
 Suppose $L$ is a flag-no-square triangulation of a closed connected manifold $N$.
 Then the Gromov boundary of $W_L$ is homeomorphic to the tree of manifolds $\mathcal X(N)$ (if $N$ is non-orientable) or $\mathcal X(N\# \overline N)$ (if $N$ is orientable).
\end{lemma}

\subsubsection{Pontryagin surfaces}\label{sec:dunce}
Fix an integer $p\geq 2$. 
Let $M_p$ be the mapping cylinder of the degree $p$ cover $f_p:S^1\to S^1,f(z)=z^p$, i.e.,
$$ M_p =\left( ([0,1]\times S^1) \sqcup S^1 \right) / \sim,$$
where $\sim$ is the equivalence relation generated by $(1,x)\sim f(x)$ for all $x\in S^1$.
The \textit{boundary} of $M_p$ is the circle given by the image of $\{0\}\times S^1$ in $M_p$. 
Note that $M_2$ is a M\"obius band, while  for $p\geq 3$ the space $M_p$ is neither a surface nor a planar set.
The \textit{core} of $M_p$  is the circle given by the image of $\{1\}\times S^1$  in $M_p$. 
Note that collapsing the core of $M_p$ to a point results in a 2-disk.

We now define a sequence $\Sigma_k$ of 2-dimensional spaces as follows: we set $\Sigma_1=S^2$, and $\Sigma_{k+1}$ is obtained by triangulating $\Sigma_k$ and then replacing each 2-simplex of $\Sigma_k$ with a copy of $M_p$, gluing along their boundaries.
Collapsing the core of each copy of $M_p$ to a point provides a natural collapsing map $q_k^{k+1}:\Sigma_{k+1}\to \Sigma_k$.
This defines an inverse system, and the \textit{Pontryagin surface} $\Pi_p$ is defined as the associated inverse limit.
Since any two triangulations of $\Sigma_k$ admit a common refinement (see  \cite{Hauptbook} and references therein), inverse limit does not depend on the triangulations chosen at each step.
We note that $\Pi_p$ is not a manifold, but $\Pi_2$ is the nonorientable Pontryagin sphere, i.e., the  tree of manifolds $\mathcal X(\mathbb{RP}^2)$.
According to \cite[Example 1.9, Theorem 1.10]{DR05}, the Pontryagin surfaces $\Pi_p$ do not embed in $\rr^3$.

Pontryagin surfaces appear as boundaries of Coxeter groups as follows.
Let $\widehat M_p$ be the $p$-fold dunce hat, i.e., the 2-complex obtained by gluing a disk to the boundary of $M_p$. 
The space $\widehat M_p$ is a Moore space 
$M(\zz/p\zz,1)$.

\begin{lemma}[{Proof of \cite[Corollary 3]{DR97}}]\label{lem:pontryagin}
 Let $p\geq 2$ be prime, and let $L$ be a flag-no-square triangulation of~$\widehat M_p$.
 Then the Gromov boundary of $W_L$ is homeomorphic to the Pontryagin surface $\Pi_p$.
\end{lemma}

\subsubsection{Quasiconvex subgroups with limit set a Menger curve}
In \cite{DLMR26}, we showed that if $L$ is a flag-no-square triangulation of a surface of positive genus, then $W_L$ contains a quasiconvex subgroup with limit set homeomorphic to a Menger curve.
We now prove a similar criterion for triangulations of $\widehat M_p$, which applies to the triangulations in Figure~\ref{fig:dunce_hat} below.
The argument is very similar to that in \cite[\S 5]{DLMR26}, to which we refer for the relevant terminology and notation. 

\begin{proposition}
\label{prop:menger}
    Let $L$ be a triangulation of $\widehat M_p$ with at least one vertex not in the core.
    Then $W_{L^\PS}$ admits a convex cocompact subgroup with limit set a Menger curve.
\end{proposition}
\begin{proof}
Let $v$ be a vertex of $L$ not in the core of $\widehat M_p$.
    Then $v$ gives rise to a vertex of $L^\PS$, which we still denote by  $v$, whose link is a circle and consists of vertices not in the core of $\widehat M_p$.
    Consider the full subcomplex $K$ on~$(L^\PS)^{(0)}\setminus \{v\}$.
    The standard subgroup $W_K$ is quasiconvex, and we claim the limit set of $W_K$ is a Menger curve.
    By \cite[Theorem 0.1.(2)]{DS21}, it is enough to check that $K$ is non-planar and inseparable, is not a join, and has puncture-respecting cohomological dimension $\pcd (K)=1$.

    Note that $K$ deformation retracts to the core of $\widehat M_p$, hence $K\simeq S^1$.
    In particular, we have that $K$ is connected.
    Since $K$ is a full subcomplex of $L^\PS$, we have that $K$ is automatically flag-no-square.
    The complex $K$ is also non-planar as it contains a neighborhood of the core of $\widehat M_p$.
    It follows that $K$ cannot split as a join of two subcomplexes.
    Moreover, for any (possibly empty) simplex $\sigma \subseteq K$, we have that $H^k(L\setminus \sigma)=0$ for all $k\geq 2$. But~$K\simeq S^1$, so $\pcd (K)=1$.
    Finally, the complex $K$ cannot be disconnected by removing a simplex or the suspension of a simplex, because $K$ contains a neighborhood of the core of $\widehat M_p$.
\end{proof}


\section{Proof of Theorem \ref{thm:main}}

\subsection{A particularly nice Lorentzian bilinear form}

Denote by $\varphi$ the golden ratio $\frac{1+\sqrt{5}}{2}$. For $n \geq 2$, let $B_n$ be the $(n+1) \times (n+1)$ symmetric matrix with entries given by
\begin{equation}\label{eq:Bnmatrix}
    (B_n)_{i,j} = \begin{cases} 1 & \textrm{for } i=j \leqslant n \\ -\varphi & \text{otherwise}. \end{cases}
\end{equation}

We first claim that $B_n$ has signature $(n,1)$. Indeed, we can explicitly compute the eigenvalues of $B_n$; denoting by $e_1, \ldots, e_{n+1}$ the standard basis vectors for $\mathbb{R}^{n+1}$, we have that $e_i - e_{i+1}$ is an eigenvector of $B_n$ with eigenvalue $1+\varphi$ for $i=1, \ldots, n-1$. Moreover, the vector $\sum_{i=1}^n e_i + s_\pm e_{n+1}$ is an eigenvector of $B_n$
with eigenvalue $1 - (n-1+s_\pm)\varphi$, where
\[
s_\pm
= \frac{1 - \varphi(n-2) \pm \sqrt{(\varphi n - 1)^2 + 4\varphi(1+\varphi)}}{2\varphi}.
\]
Then $1-(n-1+s_-)\varphi$ is positive, whereas $1-(n-1+s_+)\varphi$ is negative, since
\[
1 - (n-1+s_\pm)\varphi
= -\frac{1}{2}\left(\varphi n - 1 \pm \sqrt{(\varphi n - 1)^2 + 4\varphi(1+\varphi)}\right)
.
\]
It follows that $B_n$ has exactly one negative eigenvalue, and hence signature $(n,1)$.

Since $B_n$ has signature $(n,1)$, we may identify one of the two components of the level set $\{x \in \mathbb{R}^{n+1} \mid  x^TB_nx = -1\}$ with $n$-dimensional real hyperbolic space $\mathbb{H}^n$, and $\mathrm{O}'(B_n, \mathbb{R})$ with $\mathrm{Isom}(\hh^n)$, where $\mathrm{O}'(B_n, \mathbb{R})$ denotes the index-$2$ subgroup of the orthogonal group $\mathrm{O}(B_n, \mathbb{R})$ preserving $\mathbb{H}^n$. We will fix this model of $\mathbb{H}^n$ for the rest of the argument.

Note that for each spacelike vector $v$, the linear hyperplane of $\rr^{n,1}$ that is orthogonal to $v$ with respect to~$B_n$ defines an oriented geodesic hyperplane of $\hh^n$, and therefore a halfspace of $\hh^n$ (namely, that consisting of $x\in \hh^n\subseteq \rr^{n,1}$ with $x^TB_nv>0$).
Moreover, in this model if $H$ and $H'$ are  geodesic hyperplanes of $\hh^n$ determined by spacelike unit vectors $v, v' \in \rr^{n+1}$, respectively, then, in the case that $H$ intersects $H'$, the dihedral angle~$\theta$ formed by $H$ and~$H'$ satisfies
$\cos\theta = |v^TB_nv'|$,
and in the case that $H$ and $H'$ are disjoint, the distance $\delta$ between $H$ and $H'$ satisfies
$\cosh\delta = |v^TB_nv'|$.
Note that, in the latter case, one has $v^TB_nv'<-1$ precisely when the corresponding halfspaces of $\hh^n$ are neither nested nor disjoint.

Let $\tau: \mathbb{Q}(\sqrt{5}) \rightarrow \mathbb{Q}(\sqrt{5})$ be the nontrivial automorphism of $\mathbb{Q}(\sqrt{5})$, and denote by $B_n^\tau$ the Galois conjugate of $B_n$, i.e., the matrix obtained by applying $\tau$ to $B_n$ entry by entry. Using the fact that $\tau(\varphi) = 1 - \varphi$, one verifies by the above computation that $B_n^\tau$ is positive-definite. It thus follows from the Borel--Harish-Chandra theorem \cite{BHC62, MT62} that $\Delta_n := \mathrm{O}'(B_n, \mathbb{Z}[\varphi])$ is a cocompact arithmetic lattice in $\mathrm{O}'(B_n, \mathbb{R}) = \mathrm{Isom}(\hh^n)$.

\subsection{Passing to a larger complex}\label{sec:larger_complex}
Now fix $n\geq 4$, and let $L$ be a simplicial complex of dimension $\leqslant 3$ on $n$ vertices $v_1, \ldots, v_n$. We construct a right-angled reflection subgroup $\Gamma_L$ of $\Delta_n$ with nerve $L^\PS$ as follows. 
(Here~$L^\PS$ denotes the complex obtained by applying to $L$ the subdivision procedure described in \S\ref{sec:simplicial subs}.)
Let $L_n$ be the 3-skeleton of the $(n-1)$-simplex with vertices $v_1,\dots,v_n$.
We may view  $L$ as a subcomplex of $L_n$, and therefore~$L^\PS$ as a  subcomplex of $L_n^\PS$. (Note that  $L^\PS$ will then necessarily be a {\em full} subcomplex of $L_n^\PS$, since the Przytycki--\'{S}wi\k{a}tkowski subdivision of a simplex of dimension $\geq 1$ is never a simplex; see \cite[Lemma 2.10]{PS09}.) 
We will show that $\Delta_n$ contains a right-angled reflection subgroup $\Gamma_n$ whose nerve is $L_n^\PS$.
We then take $\Gamma_L$ to be the standard subgroup of $\Gamma_n$ corresponding to the subcomplex $L^\PS \subset L_n^\PS$. 

\subsection{Generating the reflection group $\Gamma_n$}

We now explain informally how to construct $\Gamma_n$. Choose a vertex of the right-angled $120$-cell $P_{120} \subset \hh^4$, and let $P_{116} \subset \hh^4$ be the $116$-sided right-angled polyhedron obtained from~$P_{120}$ by ``forgetting'' the four facets of $P_{120}$ containing that vertex. We now view $P_{116}$ as a polyhedron in~$\hh^n$, and for each $3$-simplex $\sigma$ of $L_n$, we choose a particular translate $P_\sigma$ of $P_{116}$ within $\hh^n$. Miraculously, there is a natural (and highly symmetric) way to choose the $P_\sigma$ such that they together form a polyhedron in~$\hh^n$ whose associated reflection group $\Gamma_n$ is contained in $\Delta_n$ and has nerve $L_n^\PS$.

Recall $n$ is the number of vertices of $L_n$. 
We now proceed with the formal construction of $\Gamma_n$. 
We will first associate to each vertex of $L_n^\PS$ a spacelike unit\footnote{These will be unit vectors with respect to the form $B_n$.} vector in $\mathbb{R}^{n+1}$. 
Then, we will define $\Gamma_n$ to be the subgroup of $O'(n,1)$ generated by the reflections in the geodesic hyperplanes of $\hh^n$ determined by these vectors.

The vertices of $L_n^\PS$ belong to a hierarchy: 
\begin{itemize}
\item At the zeroth level of this hierarchy are the vertices $v_1, \ldots, v_n$ that already belonged to the simplicial complex $L_n$ prior to subdivision. It is helpful to think of the $v_i$ as those vertices $v_i$ of $L_n^\PS$ such that $v_i$ lies in the interior of a $0$-simplex of $L_n$. For $i=1,\ldots,n$, we associate to $v_i$ the standard basis vector~$c_ie_i$, with $c_i=1$; see Table~\ref{table:coefficients}. (Notice that $e_i^T B_ne_j = -\varphi$ for $i \neq j$.)

\item Next in the hierarchy are the ``new'' vertices $w$ of $L_n^\PS$ such that $w$ lies in the interior of a $1$-simplex of~$L_n$. There are $\binom{n}{2}$ such vertices of $L_n^\PS$ (one for each $1$-simplex, and hence $6$ for each $3$-simplex, of $L_n$). We associate the vector $f_{ij}=f_{ji}:=c_ie_i+c_je_j +c_{n+1}e_{n+1}$ to the vertex $w$ of $L_n^\PS$ lying in the interior of the $1$-simplex of $L_n$ determined by $v_i$ and $v_j$; for the values of the coefficients, see Table \ref{table:coefficients}. (Notice that $e_m^T B_n f_{ij} = 0$ for $m = i, j$.)

\item Next are the vertices $w$ of $L_n^\PS$ such that $w$ lies in the interior of a $2$-simplex of $L_n$. 
There are $3\binom{n}{3}$ such vertices of $L_n^\PS$ (three for each $2$-simplex, and hence $12$ for each $3$-simplex, of $L_n$). 
We associate the vector $f^i_{jk}=f^i_{kj}:=c_ie_i+c_je_j+c_ke_k +c_{n+1}e_{n+1}$ to the unique vertex $w_i$ of $L_n^\PS$ lying in the interior of the $2$-simplex of $L_n$ determined by $v_i$, $v_j$, $v_k$ and adjacent to $v_i$ in $L_n^\PS$ (see  Figure~\ref{fig:PS-subdivision}); see Table \ref{table:coefficients}  for the values of the coefficients. (Notice that $e_i^T B_n f^{i}_{jk}$, $ f_{ij}^T B_n f^{i}_{jk}$, and $(f_{jk}^i)^T B_n f_{ki}^j$ all vanish.)

\item At the final level are the vertices $w$ of $L_n^\PS$ such that $w$ lies in the interior of a $3$-simplex $\sigma$ of $L_n$. 
There are $94\binom{n}{4}$ such vertices of $L_n^\PS$ ($94$ for each $3$-simplex of $L_n$).
The assignment of vectors $c_i e_i+c_j e_j+c_k e_k+c_\ell e_\ell +c_{n+1}e_{n+1}$ to the vertices of $L_n^\PS$ lying in the interior of the $3$-simplex of $L_n$ determined by $v_i$, $v_j$,~$v_k$ and~$v_\ell$ is again described in Table~\ref{table:coefficients}, but, for brevity's sake, we will not be as explicit here as in the previous steps of the hierarchy. 
Note that in this case there is less symmetry, so we have multiple rows in  Table~\ref{table:coefficients} corresponding to this level of the hierarchy.
We will nevertheless verify that the collection of vectors thus obtained gives rise to a reflection group with the correct nerve.
\end{itemize}

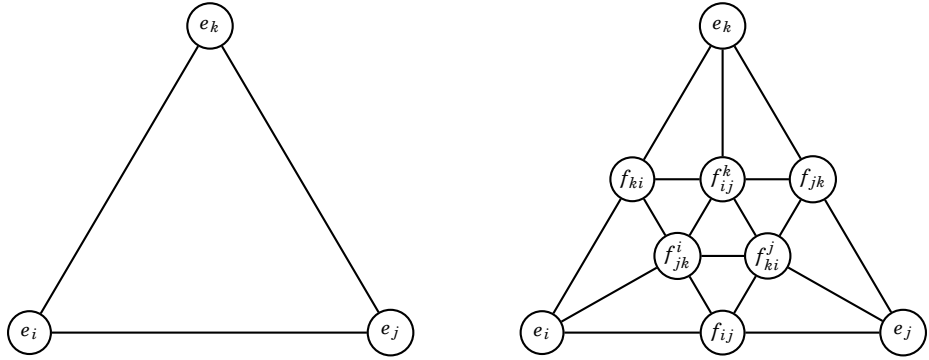
\begin{figure}[ht!]
\begin{tikzpicture}[thick,scale=0.8, every 
node/.style={transform shape}]

\node[draw,circle, inner sep=3.5pt, minimum size=3.5pt] (0i) at (0,0) {{$e_i$}};
\node[draw,circle, inner sep=3.5pt, minimum size=3.5pt] (0j) at (6,0) {{$e_{j}$}};
\node[draw,circle, inner sep=3.5pt, minimum size=3.5pt] (0k) at (3,5.1) {{$e_{k}$}};

\draw (0i) -- (0j);
\draw (0j) -- (0k);
\draw (0k) -- (0i);

\end{tikzpicture}
\quad\quad\quad\quad
\begin{tikzpicture}[thick,scale=0.8, every 
node/.style={transform shape}]

\node[draw,circle, inner sep=3.5pt, minimum size=3.5pt] (0i) at (0,0) {{$e_i$}};
\node[draw,circle, inner sep=3.5pt, minimum size=3.5pt] (0j) at (6,0) {{$e_{j}$}};
\node[draw,circle, inner sep=3.5pt, minimum size=3.5pt] (0k) at (3,5.1) {{$e_{k}$}};

\node[draw,circle, inner sep=2.0pt, minimum size=3.5pt] (1ij) at (3,0) {{$f_{ij}$}};
\node[draw,circle, inner sep=2.0pt, minimum size=3.5pt] (1jk) at (4.5,2.55) {{$f_{jk}$}};
\node[draw,circle, inner sep=2.0pt, minimum size=3.5pt] (1ki) at (1.5,2.55) {{$f_{ki}$}};

\node[draw,circle, inner sep=1.0pt, minimum size=3.5pt] (2k) at (3,2.55) {{$f_{ij}^{k}$}};
\node[draw,circle, inner sep=1.0pt, minimum size=3.5pt] (2i) at (2.25,1.275) {{$f_{jk}^{i}$}};
\node[draw,circle, inner sep=1.0pt, minimum size=3.5pt] (2j) at (3.75,1.275) {{$f_{ki}^{j}$}};

\draw (0i) -- (1ij);
\draw (0j) -- (1ij);
\draw (0j) -- (1jk);
\draw (0k) -- (1jk);
\draw (0k) -- (1ki);
\draw (0i) -- (1ki);

\draw (0i) -- (2i);
\draw (0j) -- (2j);
\draw (0k) -- (2k);

\draw (2i) -- (1ki);
\draw (2i) -- (1ij);
\draw (2j) -- (1ij);
\draw (2j) -- (1jk);
\draw (2k) -- (1ki);
\draw (2k) -- (1jk);

\draw (2i) -- (2j);
\draw (2j) -- (2k);
\draw (2k) -- (2i);

\end{tikzpicture}
\caption{(Left) A $2$-simplex $\sigma$ of $L_n$; (Right) the subdivision $\sigma^\PS$ of $\sigma$ in $L_n^\PS$.}\label{fig:PS-subdivision}
\end{figure}

\begin{table}[h!]
\begin{tabular}{@{} m{1.6cm}||c|c|c|c|c||c||m{1.9cm}}
\hline
\phantom{.} Level in \phantom{.} hierarchy       & $c_i$ & $c_j$ & $c_k$ & $c_{\ell}$ &  $c_{n+1}$ & sum  & $\#$ of permutations of $c_i$, $c_j$, $c_k$, $c_\ell$ \\
\hline\hline
\multicolumn{1}{c||}{$0$} & $1$           & $0$           & $0$           & $0$           & $0$              &  $1$  &  \multicolumn{1}{c}{$4$}   \\
\hline\hline
\multicolumn{1}{c||}{$1$} & $1$           & $1$           & $0$           & $0$           &   $-(2-\varphi)$ &  $\varphi$  &  \multicolumn{1}{c}{$6$}   \\
\hline\hline
\multicolumn{1}{c||}{$2$} & $\varphi$     & $1$           & $1$           & $0$           &   $-1$           &  $1+\varphi$  &  \multicolumn{1}{c}{$12$}   \\
\hline\hline
  & $1+\varphi$   & $\varphi$     & $1$           & $1$           &   $-2$           &  $1+2\varphi$  &  \multicolumn{1}{c}{$12$}   \\
\cline{2-8}
  & $1+\varphi$   & $1+\varphi$   & $\varphi$     & $1$           &   $-(1+\varphi)$ &  $2+2\varphi$  &  \multicolumn{1}{c}{$12$}  \\
\cline{2-8}
  & $1+\varphi$   & $1+\varphi$   & $1+\varphi$   & $1$           &   $-3$           &  $1+3\varphi$  &  \multicolumn{1}{c}{$4$}   \\
\cline{2-8}
  & $2+\varphi$   & $1+\varphi$   & $1+\varphi$   & $\varphi$     & $-(2+\varphi)$   &  $2+3\varphi$  &  \multicolumn{1}{c}{$12$}  \\ 
\cline{2-8}
  & $1+2\varphi$  & $1+\varphi$   & $1+\varphi$   & $1+\varphi$   & $-(1+2\varphi)$  &  $3+3\varphi$  &  \multicolumn{1}{c}{$4$}   \\ 
\cline{2-8}
\multicolumn{1}{c||}{$3$}  & $1+2\varphi$  & $2+\varphi$   & $1+\varphi$   & $1+\varphi$   & $-(3+\varphi)$   &  $2+4\varphi$  &  \multicolumn{1}{c}{$12$}  \\   
\cline{2-8}
  & $1+2\varphi$  & $1+2\varphi$  & $2+\varphi$   & $1+\varphi$   & $-(2+2\varphi)$  &  $3+4\varphi$  &  \multicolumn{1}{c}{$12$}  \\
\cline{2-8}
  & $2+2\varphi$  & $1+2\varphi$  & $1+2\varphi$  & $2+\varphi$   & $-(3+2\varphi)$  &  $3+5\varphi$  &  \multicolumn{1}{c}{$12$}  \\
\cline{2-8}
  & $2+2\varphi$  & $2+2\varphi$  & $1+2\varphi$  & $1+2\varphi$  & $-(2+3\varphi)$  &  $4+5\varphi$  &  \multicolumn{1}{c}{$6$}   \\
\cline{2-8}
  & $2+2\varphi$  & $2+2\varphi$  & $2+2\varphi$  & $1+2\varphi$  & $-(4+2\varphi)$  &  $3+6\varphi$  &  \multicolumn{1}{c}{$4$}   \\
\cline{2-8}
  & $1+3\varphi$  & $2+2\varphi$  & $2+2\varphi$  & $2+2\varphi$  & $-(3+3\varphi)$  &  $4+6\varphi$  &  \multicolumn{1}{c}{$4$}   \\
\hline
\end{tabular}
\caption{The coefficients of the vectors corresponding to the vertices of $L_n^\PS$. The rightmost column contains the number of vertices of each type within the subdivision of a single 3-simplex.}\label{table:coefficients}
\end{table}

\subsection{The nerve of $\Gamma_n$ is $L_n^{\PS}$}
We now proceed to check that $\Gamma_n$ is a right-angled reflection group with nerve~$L_n^{\PS}$.
In other words, we will check that the reflections associated above to the vertices of $L_n^{\PS}$ have an appropriate Gram matrix.

We will first consider the case of a pair of vertices of $L_n^\PS$ that arise from a common 3-simplex of $L_n$.
We will then consider the case of a pair of vertices of $L_n^\PS$ arising from distinct 3-simplices of $L_n$.

\subsubsection{Pairs of vertices from a single $3$-simplex of $L_n$}
Fix four distinct indices $i,j,k,\ell \in \{1,\dots,n\}$ corresponding to a single $3$-simplex $\sigma_{ijk\ell}$ of $L_n$.
Consider the following four vectors
\[
x_p = -(\varphi-1)e_m + (\varphi-1)e_{n+1} \,\textrm{ for } p=1,2,3,4,
\]
corresponding to the four facets of $P_{120}$ that we ``forgot'' in order to obtain $P_{116}$.
For $p=5,\dots, 120$, let $x_p$  be the vectors in the span of $e_i,e_j,e_k,e_\ell,e_{n+1}$ (written in that basis) assigned previously to the 116 vertices of $\sigma_{ijk\ell}^{\PS}$.
Note that the coordinates of $x_p$ are zero outside of $i,j,k,\ell,n+1$.
We now restrict to the 5-dimensional subspace of $\rr^{n,1}$ given by the span of $e_i,e_j,e_k,e_\ell,e_{n+1}$.
Note that the matrix representing the restriction of $B_n$ to the span of $e_i,e_j,e_k,e_\ell,e_{n+1}$, with respect to the latter basis, is the matrix $B_4$; see \eqref{eq:Bnmatrix}.
Moreover, we have
\[
g^{T} B_4 g =
\begin{pmatrix}
2 & 0 & 0 & 0 & 0 \\
0 & 2 & 0 & 0 & 0 \\
0 & 0 & 2 & 0 & 0 \\
0 & 0 & 0 & 2 & 0 \\
0 & 0 & 0 & 0 & -\varphi
\end{pmatrix},
\]
where
\[
g =
\begin{pmatrix}
-\varphi & -1 & 1-\varphi & 0 & \varphi \\
-\varphi & -1 & \varphi-1 & 0 & \varphi \\
-1 & -\varphi & 0 & 1-\varphi & \varphi \\
-1 & -\varphi & 0 & \varphi-1 & \varphi \\
2 & 2 & 0 & 0 & 1-2\varphi
\end{pmatrix}.
\]

Then we have $\varphi^{-1} g^{-1} x_p=(z_p,1)\in \rr^4\times \rr$, where $z_p\in \rr^4$ give the Cartesian coordinates of the vertices of a $600$-cell (the combinatorial dual of the $120$-cell) of unit radius, centered at the origin of Euclidean $4$-space and with edges of length $\varphi^{-1}$. 

More precisely, the corresponding vectors $z_p$ in Euclidean $4$-space are as follows:
\begin{itemize}
\item $8$ vectors obtained by taking permutations of
\[
(\pm 1,0,0,0);
\]

\item $16$ vectors of the form
\[
\left( \pm\frac{1}{2}, \pm\frac{1}{2}, \pm\frac{1}{2}, \pm\frac{1}{2} \right);
\]

\item the remaining $96$ vectors, obtained by taking even permutations of
\[
\left( \pm\frac{1}{2}\varphi, \pm\frac{1}{2}, \pm\frac{1}{2}\varphi^{-1}, 0 \right)
\]
\end{itemize}
(see Coxeter \cite[\S 8.7]{Coxeter48}). 

Let $z_q$ and $z_r$ be vectors in $\rr^4$ corresponding to vertices of the $600$-cell, and let $\theta$ denote the angle between them. Then, using the coordinates, one gets that $\cos \theta \in \left\{1, \frac{\varphi}{2}, \frac{1}{2}, \frac{\varphi^{-1}}{2}, 0, -\frac{\varphi^{-1}}{2}, -\frac{1}{2}, -\frac{\varphi}{2}, -1\right\}$. Hence, if $z_q$ and $z_r$ correspond to adjacent vertices of the $600$-cell, then
\[
\cos \theta = \frac{\varphi}{2}.
\]
Therefore,
\[
x_q^{T} B_4 x_r
= \varphi^{2} (\varphi^{-1} g^{-1} x_q)^{T} (g^{T} B_4 g) (\varphi^{-1} g^{-1} x_r)
= \varphi^{2} (2 \cos \theta - \varphi)
= 0.
\]
This implies that the vectors $x_q$ and $x_r$ are orthogonal with respect to the Lorentzian form $B_4$, so that the corresponding generators of $\Gamma_n$ commute. 

Moreover, if $z_q$ and $z_r$ are not adjacent, then the angle $\theta$ between $z_q$ and $z_r$ satisfies $\cos \theta \leqslant \tfrac{1}{2}$. It follows that in this case
\[
x_q^{T} B_4 x_r = \varphi^{2} (2 \cos \theta - \varphi) \leqslant -\varphi .
\]
Consequently, if the vectors $x_q$ and $x_r$ are not orthogonal, their Lorentzian inner product satisfies
\[
x_q^{T} B_4 x_r \leqslant -\varphi < -1 .
\]
This shows that the associated hyperplanes of $\hh^n$ are disjoint, and that the associated halfspaces are neither nested nor disjoint.

We have at this point verified that, for any two vertices of $L_n^\PS$ lying in a common $3$-simplex of $L_n$, the associated vectors in $\rr^{n+1}$ have the correct combinatorics, i.e., we have shown that the list of reflections associated above to vertices of $L_n^\PS$ arising from a single $3$-simplex of $L_n$ have an appropriate Gram matrix.

\subsubsection{Pairs of vertices from distinct $3$-simplices of $L_n$}

We will now have to consider pairs of vertices of $L_n^\PS$ that do not lie in a common $3$-simplex of $L_n$.
Note that necessarily such vertices are not adjacent in $L_n^\PS$, so we are going to check that the corresponding geodesic hyperplanes of $\hh^n$ have positive distance (and, moreover, that the corresponding halfspaces of $\hh^n$ are neither nested nor disjoint).

Let
\[
x = c_i e_i + c_j e_j + c_k e_k + c_\ell e_\ell + c_{n+1} e_{n+1}
\quad \text{and} \quad
y = c'_{i'} e_{i'} + c'_{j'} e_{j'} + c'_{k'} e_{k'} + c'_{\ell'} e_{\ell'} + c'_{n+1} e_{n+1}
\]
be the vectors corresponding to a pair of vertices $v$ and $w$ of $L_n^{\PS}$, respectively. 
Assume that $v$ and $w$ do not lie in a common $3$-simplex of $L_n$.
Let $\sigma_v$ (respectively, $\sigma_w$) be the simplex of $L_n$ containing $v$ (resp., $w$) in its interior.

\begin{itemize}
\item[Case 1:] $\sigma_v$ and $\sigma_w$ do not intersect. In this case,
\[
\{i,j,k,\ell\} \cap \{i',j',k',\ell'\} = \varnothing,
\]
and
\[
x^T B_n y = -\varphi \bigl(c_i + c_j + c_k + c_\ell + c_{n+1}\bigr)
        \bigl(c'_{i'} + c'_{j'} + c'_{k'} + c'_{\ell'} + c'_{n+1}\bigr)
\leqslant -\varphi < -1.
\]
since $c_i + c_j + c_k + c_\ell + c_{n+1} \geqslant 1$, as one can easily check in Table~\ref{table:coefficients}. 

\item[Case 2:] The intersection of $\sigma_v$ and $\sigma_w$ is a $0$-simplex of $L_n$.
Up to permuting the indices,  it suffices to consider the case where $i = i'$ and
\[
\{j,k,\ell\} \cap \{j',k',\ell'\} = \varnothing.
\]
In this case, we have
\begin{equation}\label{eq:case2}
x^T B_n y = -\varphi \bigl(c_i + c_j + c_k + c_\ell + c_{n+1}\bigr)
        \bigl(c'_{i} + c'_{j'} + c'_{k'} + c'_{\ell'} + c'_{n+1}\bigr)
        + (1+\varphi) c_i c'_{i}.
\end{equation}

If $c_i = 0$ or $c'_i = 0$, then \eqref{eq:case2} $\leqslant -\varphi$. Hence, we may assume that $c_i, c'_i \neq 0$, that is, $c_i, c'_i \geqslant 1$ (see Table~\ref{table:coefficients}).
Since $v$ and $w$ do not lie in a common $3$-simplex of $L_n$, and since $v$ (respectively,  $w$) lies in the interior of $\sigma_v$ (resp., $\sigma_w$), the level of each of $v$ and $w$ in the hierarchy is greater than $0$. 
Hence, one can check using Table~\ref{table:coefficients} and the identity $\varphi^2=\varphi +1$ that 
\[
\varphi c_i \leqslant c_i + c_j + c_k + c_\ell + c_{n+1}.
\]
It follows that \eqref{eq:case2} is bounded above by
\[
-\varphi \cdot (\varphi c_i)(\varphi c'_i) + (1+\varphi) c_i c'_i
= -\varphi c_i c'_i
\leqslant -\varphi < -1.
\]

\item[Case 3:] The intersection of $\sigma_v$ and $\sigma_w$ is a $1$-simplex of $L_n$.

It suffices to consider the case where $i = i'$, $j = j'$, and
\[
\{k,\ell\} \cap \{k',\ell'\} = \varnothing.
\]
In this case, we have
\begin{equation}\label{eq:case3}
x^T B_n y = -\varphi \bigl(c_i + c_j + c_k + c_\ell + c_{n+1}\bigr)
        \bigl(c'_i + c'_j + c'_{k'} + c'_{\ell'} + c'_{n+1}\bigr)
        + (1+\varphi)\bigl(c_i c'_i + c_j c'_j\bigr).
\end{equation}
By the Cauchy--Schwarz inequality, \eqref{eq:case3} is bounded above by
\begin{equation}\label{eq:case3-2}
-\varphi \bigl(c_i + c_j + c_k + c_\ell + c_{n+1}\bigr)
        \bigl(c'_i + c'_j + c'_{k'} + c'_{\ell'} + c'_{n+1}\bigr)
        + (1+\varphi)\sqrt{c_i^2 + c_j^2}\,\sqrt{(c'_i)^2 + (c'_j)^2}.
\end{equation}

As in the previous case, since $v$ and $w$ do not lie in a common $3$-simplex of $L_n$, their level in the hierarchy is greater than $1$. Looking at Table~\ref{table:coefficients}, one sees that the worst-case scenario is that $c_i$, $c'_i$ are the first coefficients; that $c_j$, $c'_j$ are the second coefficients; that $c_k$, $c'_{k'}$ are the third coefficients; and that $c_{\ell}$, $c'_{\ell'}$ are the last coefficients in Table~\ref{table:coefficients}. An exhaustive computation, provided in the SageMath file (\emph{exhaustive-PS-subdivision-Limit-Set}.ipynb) and available on the webpage \cite{computations}, of the $\binom{13}{2} = 78$ possibilities then shows that 
\[
\eqref{eq:case3-2} \leqslant -\varphi < -1.
\]

\item[Case 4:] The intersection of $\sigma_v$ and $\sigma_w$ is a $2$-simplex of $L_n$.

It suffices to consider the case where $i = i'$, $j = j'$, $k = k'$, and $\ell \neq \ell'$.
In this case, we have
\begin{equation}\label{eq:case4}
x^T B_n y = -\varphi \bigl(c_i + c_j + c_k + c_\ell + c_{n+1}\bigr)
        \bigl(c'_i + c'_j + c'_k + c'_{\ell'} + c'_{n+1}\bigr)
        + (1+\varphi)\bigl(c_i c'_i + c_j c'_j + c_k c'_k\bigr).
\end{equation}

Again, by the Cauchy--Schwarz inequality, \eqref{eq:case4} is bounded above by
\begin{equation}\label{eq:case4-2}
-\varphi \bigl(c_i + c_j + c_k + c_\ell + c_{n+1}\bigr)
        \bigl(c'_i + c'_j + c'_k + c'_{\ell'} + c'_{n+1}\bigr)
        + (1+\varphi)\sqrt{c_i^2 + c_j^2 + c_k^2}\,
        \sqrt{(c'_i)^2 + (c'_j)^2 + (c'_k)^2}.
\end{equation}

As in the previous cases, since $v$ and $w$ do not lie in a common $3$-simplex of $L_n$, their level in the hierarchy is equal to $3$.   
Looking at Table~\ref{table:coefficients}, one sees that the worst-case scenario is that $c_i$, $c'_i$ are the first coefficients; that $c_j$, $c'_j$ are the second coefficients; that $c_k$, $c'_{k'}$ are the third coefficients; and that $c_{\ell}$, $c'_{\ell'}$ are the last coefficients in Table~\ref{table:coefficients}. An exhaustive computation, provided in the SageMath file (\emph{exhaustive-PS-subdivision-Limit-Set}.ipynb) and available on the webpage \cite{computations}, of the $\binom{12}{2} = 66$ possibilities then shows that 
\[
\eqref{eq:case4-2} \leqslant -\varphi < -1.
\]
\end{itemize}

Denoting by $x_v \in \rr^{n+1}$ the unit vector we have associated to the vertex $v \in L_n^{\# (0)}$, we have checked that, for every pair of vertices $v,w \in L_n^{\# (0)}$, one has that $B_n(x_v,x_w) = 0$ if and only if $v$ and $w$ are adjacent in $L_n^{\#}$; and that, if~$v$ and $w$ are not adjacent, then $B_n(x_v,x_w) < -1$. Hence, the group $\Gamma_n$ generated by the reflections in the linear hyperplanes of $\rr^{n+1}$ orthogonal to the $x_v$ with respect to the form $B_n$ is indeed a reflection group in $\mathrm{Isom} (\hh^n)$ whose nerve is $L_n^\PS$.

\begin{remark}
The strict inequality $B_n(x_v,x_w) < -1$ whenever $v$ and $w$ are not adjacent moreover shows that $\Gamma_n$ is convex cocompact (see, for instance, \cite[Theorem~4.7]{DH13}). 
Since the reflection group $\Gamma_n$ is automatically geometrically finite, convex cocompactness will also follow once we establish in \S\ref{subsection:containment} that $\Gamma_n$ is a subgroup of the cocompact lattice $\Delta_n$).
We obtain in particular that $\Gamma_n$ is Gromov-hyperbolic, and hence recover the result of Przytycki--\'{S}wi\k{a}tkowski \cite[Proposition~2.13]{PS09} that the complex $L_n^\PS$ is flag-no-square.
\end{remark}

\subsection{$\Gamma_n$ is a subgroup of $\Delta_n$}\label{subsection:containment}

To see that $\Gamma_n \subset \Delta_n$, note that, for each vertex $v$ of $L_n^\PS$ with associated unit vector $x_v \in \rr^{n+1}$ as specified above, the reflection in the linear hyperplane of $\rr^{n+1}$ orthogonal to $x_v$ with respect to the form $B_n$ is given by 
\[x \mapsto x - (x^T B_n x_v)x_v\] 
for $x \in \mathbb{R}^{n+1}$, and this map has image in $\Delta_n=\mathrm{O}'(B_n, \mathbb{Z}[\varphi])$  since the entries of $B_n$ and $x_v$ all lie in $\mathbb{Z}[\varphi]$; see Table~\ref{table:coefficients}.

\subsection{$\Gamma_L$ is  Zariski-dense in $\mathrm{O}(B_n,\mathbb{R})$}
Finally, recall from \S\ref{sec:larger_complex} that $L\subseteq L_n$ and that $\Gamma_L$ is the standard subgroup of $\Gamma_n$ corresponding to the full subcomplex  $L^\PS\subseteq L_n^\PS$.

We now show that the reflection group $\Gamma_L$ is Zariski-dense in $\mathrm{O}(B_n,\mathbb{R})$.
We first show that 
if $(L^{\PS})^{(1)}$ is the join of two disjoint subsets $V_1$ and $V_2$ of $(L^{\PS})^{(0)}$ whose union is $(L^{\PS})^{(0)}$, then either $V_1$ or $V_2$ is empty. First, since no two vertices in $L^{(0)}$ are adjacent in $(L^{\PS})^{(1)}$, the vertices of $L^{(0)}$ all lie in precisely one of the $V_i$, say $V_1$. Now a hypothetical vertex $v$ in $V_2$ that is adjacent to every vertex in $V_1$, and hence to every vertex in $L^{(0)}$, cannot lie in the interior of a $2$- or $3$-simplex of $L$. Moreover, since $n \geqslant 3$, such a vertex $v$ also cannot lie in the interior of a $1$-simplex of $L$. We conclude that $V_2$ must indeed be empty.

Now since $(L^{\PS})^{(1)}$ is not a join, any $\Gamma_L$-invariant subspace of $\mathbb{R}^{n+1}$ either contains the span $U$ of the unit vectors that we have associated to $(L^{\PS})^{(0)}$, or is contained in the intersection $U'$ of their orthogonal complements with respect to the form $B_n$; see, e.g., Vinberg \cite[Proposition~19]{Vinberg72} or \cite[Proposition~3.23]{DGKLM}. 

Since $L$ has dimension $d \geqslant 1$, there are vertices in a positive level of the hierarchy, so we have that $U$ contains $e_1,\dots,e_n$ as well as a vector of the form $\sum_{i=1}^{n+1} c_i e_i$ with $c_{n+1} \neq 0$, and hence coincides with $\mathbb{R}^{n+1}$. Moreover, since $B_n$ is nondegenerate, we have that $U'$ is trivial. Therefore, the action of $\Gamma_L$ on $\mathbb{R}^{n+1}$ is irreducible.

Since $B_n$ is Lorentzian, the Zariski closure of any irreducible subgroup of $\mathrm{O}(B_n,\mathbb{R})$ contains $\mathrm{SO}(B_n,\mathbb{R})$; see, e.g., \cite[Proposition~1]{benoist_harpe}. It follows that $\Gamma_L$ is Zariski-dense in $\mathrm{O}(B_n,\mathbb{R})$.

This completes the proof of Theorem~\ref{thm:main}.

\section{Proofs of Corollaries~\ref{cor:Jakobsche_space},~\ref{cor:manifold},~ and~\ref{cor:pontryagin}}

\begin{proof}[Proof of Corollary~\ref{cor:Jakobsche_space}]
Let $L$ be a triangulation of $N$ with $n$ vertices. By Theorem~\ref{thm:main}, there is a convex cocompact right-angled reflection group $\Gamma_L < \mathrm{Isom}(\mathbb{H}^n)$ whose nerve is the flag-no-square triangulation $L^\PS$ of $N$. By Lemma~\ref{lem:tree}, the Gromov boundary of~$\Gamma_L$ is homeomorphic to $\mathcal X(N)$ (if $N$ is non-orientable) or $\mathcal X(N\# \overline N)$ (if $N$ is orientable), and hence so is the limit set $\Lambda_{\Gamma_L}$ by convex cocompactness of $\Gamma_L$.
\end{proof}

\begin{proof}[Proof of Corollary~\ref{cor:manifold}]
Item~(\ref{item:h6}) follows Corollary~\ref{cor:Jakobsche_space}, together with the existence of a 6-vertex triangulation of the real projective plane. Item~(\ref{item:h7}) follows again from Corollary~\ref{cor:Jakobsche_space}, together with the existence of a 7-vertex triangulation of the torus. Item~(\ref{item:h16}) follows from  Theorem~\ref{thm:main} and Lemma~\ref{lem:RACG homology 3-sphere}, together with the existence of a $16$-vertex triangulation of the Poincar\'e homology $3$-sphere due to Bj\"{o}rner and Lutz~\cite{BL00}.
\end{proof}

The latter authors conjecture \cite[Conjecture 6]{BL00} that 16 is the minimal number of vertices in any  triangulation of the Poincar\'e homology sphere. 
It is known \cite{BD05} that at least $12$ vertices are needed to triangulate any homology $3$-sphere differing from~$S^3$.

\begin{proof}[Proof of Corollary~\ref{cor:pontryagin}]
     This follows from Theorem~\ref{thm:main} and Lemma~\ref{lem:pontryagin}, together with the existence of a triangulation of the $p$-fold dunce hat $\widehat M_p$ with $\big\lceil \tfrac{3p}{2} \big\rceil + 3$ vertices. Figure~\ref{fig:dunce_hat} illustrates the cases $p=2$, $p=3$, and $p=5$.
\end{proof}
Note that the leftmost triangulation in Figure~\ref{fig:dunce_hat} is the  triangulation of the real projective plane $\mathbb{RP}^2=\widehat M_2$ arising from the complete graph on $6$ vertices, and $6$ is the minimal possible number of vertices in any triangulation of $\mathbb{RP}^2$; see~\cite{BA82}. 

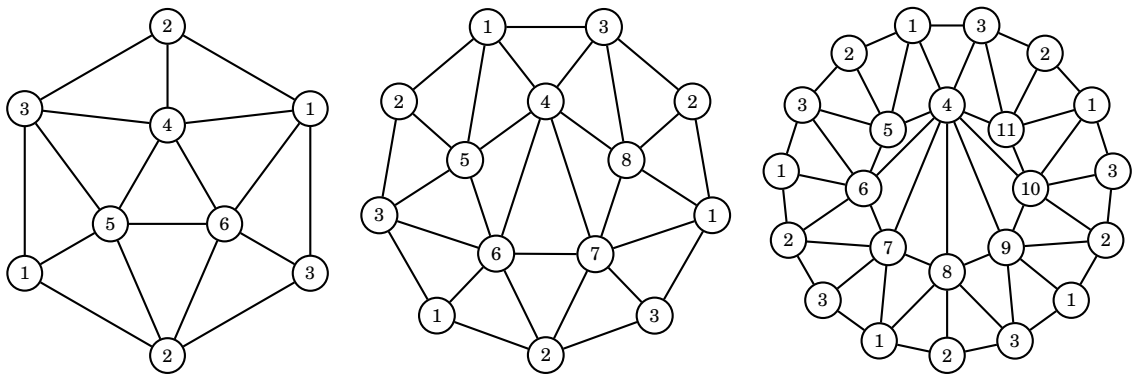
\begin{figure}[ht!]
\begin{tikzpicture}[thick,scale=0.73, every 
node/.style={transform shape}]

\node[draw,circle, inner sep=3.5pt, minimum size=3.5pt] (11) at (0,-3) {{$2$}};
\node[draw,circle, inner sep=3.5pt, minimum size=3.5pt] (12) at (-2.60,-1.5) {{$1$}};
\node[draw,circle, inner sep=3.5pt, minimum size=3.5pt] (13) at (-2.6,1.5) {{$3$}};
\node[draw,circle, inner sep=3.5pt, minimum size=3.5pt] (21) at (0,3) {{$2$}};
\node[draw,circle, inner sep=3.5pt, minimum size=3.5pt] (22) at (2.60,1.5) {{$1$}};
\node[draw,circle, inner sep=3.5pt, minimum size=3.5pt] (23) at (2.60,-1.5) {{$3$}};

\node[draw,circle, inner sep=3.5pt, minimum size=3.5pt] (04) at (0,1.2) {{$4$}};
\node[draw,circle, inner sep=3.5pt, minimum size=3.5pt] (05) at (1.04,-0.6) {{$6$}};
\node[draw,circle, inner sep=3.5pt, minimum size=3.5pt] (06) at (-1.04,-0.6) {{$5$}};

\draw (11) -- (12);
\draw (12) -- (13);
\draw (13) -- (21);
\draw (21) -- (22);
\draw (22) -- (23);
\draw (23) -- (11);

\draw (04) -- (13);
\draw (04) -- (21);
\draw (04) -- (22);
\draw (05) -- (22);
\draw (05) -- (23);
\draw (05) -- (11);
\draw (06) -- (11);
\draw (06) -- (12);
\draw (06) -- (13);

\draw (04) -- (05);
\draw (05) -- (06);
\draw (06) -- (04);
\end{tikzpicture}
\quad
\begin{tikzpicture}[thick,scale=0.75, every node/.style={transform shape}]
\node[draw,circle, inner sep=3.5pt, minimum size=3.5pt] (12) at (0,-3) {{$2$}};
\node[draw,circle, inner sep=3.5pt, minimum size=3.5pt] (13) at (1.93,-2.30) {{$3$}};
\node[draw,circle, inner sep=3.5pt, minimum size=3.5pt] (21) at (2.95,-0.52) {{$1$}};
\node[draw,circle, inner sep=3.5pt, minimum size=3.5pt] (22) at (2.60,1.5) {{$2$}};
\node[draw,circle, inner sep=3.5pt, minimum size=3.5pt] (23) at (1.03,2.82) {{$3$}};
\node[draw,circle, inner sep=3.5pt, minimum size=3.5pt] (31) at (-1.03,2.82) {{$1$}};
\node[draw,circle, inner sep=3.5pt, minimum size=3.5pt] (32) at (-2.60,1.5) {{$2$}};
\node[draw,circle, inner sep=3.5pt, minimum size=3.5pt] (33) at (-2.95,-0.52) {{$3$}};
\node[draw,circle, inner sep=3.5pt, minimum size=3.5pt] (11) at (-1.93,-2.30) {{$1$}};

\node[draw,circle, inner sep=3.5pt, minimum size=3.5pt] (04) at (0.0,1.5) {{$4$}};
\node[draw,circle, inner sep=3.5pt, minimum size=3.5pt] (05) at (-1.43,0.46) {{$5$}};
\node[draw,circle, inner sep=3.5pt, minimum size=3.5pt] (06) at (-0.88,-1.2) {{$6$}};
\node[draw,circle, inner sep=3.5pt, minimum size=3.5pt] (07) at (0.88,-1.21) {{$7$}};
\node[draw,circle, inner sep=3.5pt, minimum size=3.5pt] (08) at (1.43,0.46) {{$8$}};

\draw (11) -- (12);
\draw (12) -- (13);
\draw (13) -- (21);
\draw (21) -- (22);
\draw (22) -- (23);
\draw (23) -- (31);
\draw (31) -- (32);
\draw (32) -- (33);
\draw (33) -- (11);

\draw (04) -- (05);
\draw (05) -- (06);
\draw (06) -- (07);
\draw (07) -- (08);
\draw (08) -- (04);

\draw (06) -- (33);
\draw (06) -- (11);
\draw (06) -- (12);
\draw (07) -- (12);
\draw (07) -- (13);
\draw (07) -- (21);
\draw (08) -- (21);
\draw (08) -- (22);
\draw (08) -- (23);
\draw (04) -- (23);
\draw (04) -- (31);
\draw (05) -- (31);
\draw (05) -- (32);
\draw (05) -- (33);

\draw (04) -- (06);
\draw (04) -- (07);

\end{tikzpicture}
\quad
\begin{tikzpicture}[thick,scale=0.74, every node/.style={transform shape}]

\node[draw,circle, inner sep=3.5pt, minimum size=3.5pt] (12) at (0,-3) {{$2$}};
\node[draw,circle, inner sep=3.5pt, minimum size=3.5pt] (13) at (1.22,-2.74) {{$3$}};
\node[draw,circle, inner sep=3.5pt, minimum size=3.5pt] (21) at (2.23,-2.01) {{$1$}};
\node[draw,circle, inner sep=3.5pt, minimum size=3.5pt] (22) at (2.85,-0.93) {{$2$}};
\node[draw,circle, inner sep=3.5pt, minimum size=3.5pt] (23) at (2.98,0.31) {{$3$}};
\node[draw,circle, inner sep=3.5pt, minimum size=3.5pt] (31) at (2.60,1.5) {{$1$}};
\node[draw,circle, inner sep=3.5pt, minimum size=3.5pt] (32) at (1.76,2.43) {{$2$}};
\node[draw,circle, inner sep=3.5pt, minimum size=3.5pt] (33) at (0.62,2.93) {{$3$}};
\node[draw,circle, inner sep=3.5pt, minimum size=3.5pt] (41) at (-0.62,2.93) {{$1$}};
\node[draw,circle, inner sep=3.5pt, minimum size=3.5pt] (42) at (-1.76,2.43) {{$2$}};
\node[draw,circle, inner sep=3.5pt, minimum size=3.5pt] (43) at (-2.60,1.5) {{$3$}};
\node[draw,circle, inner sep=3.5pt, minimum size=3.5pt] (51) at (-2.98,0.31) {{$1$}};
\node[draw,circle, inner sep=3.5pt, minimum size=3.5pt] (52) at (-2.85,-0.93) {{$2$}};
\node[draw,circle, inner sep=3.5pt, minimum size=3.5pt] (53) at (-2.23,-2.01) {{$3$}};
\node[draw,circle, inner sep=3.5pt, minimum size=3.5pt] (11) at (-1.22,-2.74) {{$1$}};

\node[draw,circle, inner sep=3.5pt, minimum size=3.5pt] (08) at (0,-1.5) {{$8$}};
\node[draw,circle, inner sep=3.5pt, minimum size=3.5pt] (09) at (1.06,-1.06) {{$9$}};
\node[draw,circle, inner sep=2.1pt, minimum size=3.5pt] (010) at (1.5,0) {{$10$}};
\node[draw,circle, inner sep=2.1pt, minimum size=3.5pt] (011) at (1.06,1.06) {{$11$}};
\node[draw,circle, inner sep=3.5pt, minimum size=3.5pt] (04) at (0.0,1.5) {{$4$}};
\node[draw,circle, inner sep=3.5pt, minimum size=3.5pt] (05) at (-1.06,1.06) {{$5$}};
\node[draw,circle, inner sep=3.5pt, minimum size=3.5pt] (06) at (-1.5,0) {{$6$}};
\node[draw,circle, inner sep=3.5pt, minimum size=3.5pt] (07) at (-1.06,-1.06) {{$7$}};

\draw (11) -- (12);
\draw (12) -- (13);
\draw (13) -- (21);
\draw (21) -- (22);
\draw (22) -- (23);
\draw (23) -- (31);
\draw (31) -- (32);
\draw (32) -- (33);
\draw (33) -- (41);
\draw (41) -- (42);
\draw (42) -- (43);
\draw (43) -- (51);
\draw (51) -- (52);
\draw (52) -- (53);
\draw (53) -- (11);

\draw (04) -- (05);
\draw (05) -- (06);
\draw (06) -- (07);
\draw (07) -- (08);
\draw (08) -- (09);
\draw (09) -- (010);
\draw (010) -- (011);
\draw (011) -- (04);

\draw (08) -- (11);
\draw (08) -- (12);
\draw (08) -- (13);
\draw (09) -- (13);
\draw (09) -- (21);
\draw (09) -- (22);
\draw (010) -- (22);
\draw (010) -- (23);
\draw (010) -- (31);
\draw (011) -- (31);
\draw (011) -- (32);
\draw (011) -- (33);
\draw (04) -- (33);
\draw (04) -- (41);
\draw (05) -- (41);
\draw (05) -- (42);
\draw (05) -- (43);
\draw (06) -- (43);
\draw (06) -- (51);
\draw (06) -- (52);
\draw (07) -- (52);
\draw (07) -- (53);
\draw (07) -- (11);

\draw (04) -- (06);
\draw (04) -- (07);
\draw (04) -- (08);
\draw (04) -- (09);
\draw (04) -- (010);
\end{tikzpicture}
\caption{Triangulations of $\widehat M_2$, $\widehat M_3$, and $\widehat M_5$, shown from left to right.}\label{fig:dunce_hat}
\end{figure}

\printbibliography

@article{DGKLM,
 author = {Danciger, Jeffrey and Gu{\'e}ritaud, Fran{\c{c}}ois and Kassel, Fanny and Lee, Gye-Seon and Marquis, Ludovic},
 title = {Convex cocompactness for {Coxeter} groups},
 fjournal = {Journal of the European Mathematical Society (JEMS)},
 journal = {J. Eur. Math. Soc. (JEMS)},
 issn = {1435-9855},
 volume = {27},
 number = {1},
 pages = {119--181},
 year = {2025},
 language = {English},
 doi = {10.4171/JEMS/1539},
 zbMATH = {7981648},
 Zbl = {1561.22009}
}

@article{Vinberg72,
 author = {Vinberg, E. B.},
 title = {Discrete linear groups generated by reflections},
 fjournal = {Mathematics of the USSR. Izvestiya},
 journal = {Math. USSR, Izv.},
 issn = {0025-5726},
 volume = {5},
 pages = {1083--1119},
 year = {1972},
 language = {English},
 doi = {10.1070/IM1971v005n05ABEH001203},
 zbMATH = {3404520},
 Zbl = {0256.20067}
}

@article {RA03,
    AUTHOR = {Radcliffe, David G.},
     TITLE = {Rigidity of graph products of groups},
   JOURNAL = {Algebr. Geom. Topol.},
  FJOURNAL = {Algebraic \& Geometric Topology},
    VOLUME = {3},
      YEAR = {2003},
     PAGES = {1079--1088},
      ISSN = {1472-2747,1472-2739},
   MRCLASS = {20F65},
  MRNUMBER = {2012965},
MRREVIEWER = {Vladimir\ N.\ Bezverkhni\u i},
       DOI = {10.2140/agt.2003.3.1079},
       URL = {https://doi.org/10.2140/agt.2003.3.1079},
}

@article{BL00,
 author = {Bj{\"o}rner, Anders and Lutz, Frank H.},
 title = {Simplicial manifolds, bistellar flips and a 16-vertex triangulation of the {Poincar{\'e}} homology 3-sphere.},
 fjournal = {Experimental Mathematics},
 journal = {Exp. Math.},
 issn = {1058-6458},
 volume = {9},
 number = {2},
 pages = {275--289},
 year = {2000},
 language = {English},
 doi = {10.1080/10586458.2000.10504652},
 url = {https://eudml.org/doc/226563},
 zbMATH = {1647649},
 Zbl = {1101.57306}
}

@article{Ma25,
    AUTHOR = {Ma, Jiming},
     TITLE = {The {U}niversal 2{D} {M}enger {S}pace as {L}imit {S}ets of
              {H}igher-{D}imensional {K}leinian {G}roups},
   JOURNAL = {Int. Math. Res. Not. IMRN},
  FJOURNAL = {International Mathematics Research Notices. IMRN},
      YEAR = {2026},
    NUMBER = {4},
     PAGES = {Paper No. rnag026},
      ISSN = {1073-7928,1687-0247},
   MRCLASS = {22E40 (20H10 57K32)},
  MRNUMBER = {5031455},
       DOI = {10.1093/imrn/rnag026},
       URL = {https://doi.org/10.1093/imrn/rnag026},
}

@ARTICLE{MZ25,
       author = {{Ma}, Jiming and {Zheng}, Fangting},
        title = "{Discrete embeddings of hyperbolic groups with Pontryagin surfaces as boundaries}",
        year = 2025,
        note = {In preparation}
}

@misc{Coxeter48,
 author = {Coxeter, H. S. M.},
 title = {Regular polytopes},
 year = {1948},
 language = {English},
 howpublished = {London: {Methuen} \& {Co}., {Ltd}., {XX}, 322 p. 8 plates. (1948).},
 zbMATH = {3048077},
 Zbl = {0031.06502}
}

@article {MT62,
    AUTHOR = {Mostow, G. D. and Tamagawa, T.},
     TITLE = {On the compactness of arithmetically defined homogeneous
              spaces},
   JOURNAL = {Ann. of Math. (2)},
  FJOURNAL = {Annals of Mathematics. Second Series},
    VOLUME = {76},
      YEAR = {1962},
     PAGES = {446--463},
      ISSN = {0003-486X},
   MRCLASS = {22.55 (20.65)},
  MRNUMBER = {141672},
MRREVIEWER = {T.\ Ono},
       DOI = {10.2307/1970368},
       URL = {https://doi.org/10.2307/1970368},
}

@article {BHC62,
    AUTHOR = {Borel, Armand and Harish-Chandra},
     TITLE = {Arithmetic subgroups of algebraic groups},
   JOURNAL = {Ann. of Math. (2)},
  FJOURNAL = {Annals of Mathematics. Second Series},
    VOLUME = {75},
      YEAR = {1962},
     PAGES = {485--535},
      ISSN = {0003-486X},
   MRCLASS = {20.65 (14.50)},
  MRNUMBER = {147566},
MRREVIEWER = {P.\ Cartier},
       DOI = {10.2307/1970210},
       URL = {https://doi.org/10.2307/1970210},
}

@incollection {Kapovich08,
    AUTHOR = {Kapovich, Michael},
     TITLE = {Kleinian groups in higher dimensions},
 BOOKTITLE = {Geometry and dynamics of groups and spaces},
    SERIES = {Progr. Math.},
    VOLUME = {265},
     PAGES = {487--564},
 PUBLISHER = {Birkh\"auser, Basel},
      YEAR = {2008},
      ISBN = {978-3-7643-8607-8},
   MRCLASS = {30F40 (20F67 57N16)},
  MRNUMBER = {2402415},
MRREVIEWER = {Michel\ Coornaert},
       DOI = {10.1007/978-3-7643-8608-5\_13},
       URL = {https://doi.org/10.1007/978-3-7643-8608-5_13},
}

@article {FR82,
    AUTHOR = {Freedman, Michael Hartley},
     TITLE = {The topology of four-dimensional manifolds},
   JOURNAL = {J. Differential Geometry},
  FJOURNAL = {Journal of Differential Geometry},
    VOLUME = {17},
      YEAR = {1982},
    NUMBER = {3},
     PAGES = {357--453},
      ISSN = {0022-040X,1945-743X},
   MRCLASS = {57N12 (57R80 57R99)},
  MRNUMBER = {679066},
MRREVIEWER = {John\ J.\ Walsh},
       URL = {http://projecteuclid.org/euclid.jdg/1214437136},
}

@article {DR97,
    AUTHOR = {Dranishnikov, A. N.},
     TITLE = {On the virtual cohomological dimensions of {C}oxeter groups},
   JOURNAL = {Proc. Amer. Math. Soc.},
  FJOURNAL = {Proceedings of the American Mathematical Society},
    VOLUME = {125},
      YEAR = {1997},
    NUMBER = {7},
     PAGES = {1885--1891},
      ISSN = {0002-9939,1088-6826},
   MRCLASS = {55M10 (20F55 20J05 57S30)},
  MRNUMBER = {1422863},
MRREVIEWER = {Alexander\ I.\ Suciu},
       DOI = {10.1090/S0002-9939-97-04106-3},
       URL = {https://doi-org.proxy.binghamton.edu/10.1090/S0002-9939-97-04106-3},
}

@article {DLMR26,
    author = {{Douba}, Sami and {Lee}, Gye-Seon and {Marquis}, Ludovic and {Ruffoni}, Lorenzo},
        title = "{Convex cocompact groups in real hyperbolic spaces with limit set a Pontryagin sphere}",
   JOURNAL = {Bull. Lond. Math. Soc.},
  FJOURNAL = {Bulletin of the London Mathematical Society},
    VOLUME = {58},
      YEAR = {2026},
    NUMBER = {4},
       DOI = {10.1112/blms.70319},
       URL = {https://doi.org/10.1112/blms.70319},
}

@article {SU79,
    AUTHOR = {Sullivan, Dennis},
     TITLE = {The density at infinity of a discrete group of hyperbolic
              motions},
   JOURNAL = {Inst. Hautes \'Etudes Sci. Publ. Math.},
  FJOURNAL = {Institut des Hautes \'Etudes Scientifiques. Publications
              Math\'ematiques},
    NUMBER = {50},
      YEAR = {1979},
     PAGES = {171--202},
      ISSN = {0073-8301,1618-1913},
   MRCLASS = {58F17 (22E40 28C10 30C85)},
  MRNUMBER = {556586},
MRREVIEWER = {Troels\ J\o rgensen},
       URL = {http://www.numdam.org/item?id=PMIHES_1979__50__171_0},
}

@article {DH13,
    AUTHOR = {Desgroseilliers, Marc and Haglund, Fr\'{e}d\'{e}ric},
     TITLE = {On some convex cocompact groups in real hyperbolic space},
   JOURNAL = {Geom. Topol.},
  FJOURNAL = {Geometry \& Topology},
    VOLUME = {17},
      YEAR = {2013},
    NUMBER = {4},
     PAGES = {2431--2484},
      ISSN = {1465-3060,1364-0380},
   MRCLASS = {22E40 (20F67 51F15 53C23 57M20)},
  MRNUMBER = {3110583},
MRREVIEWER = {Shihai\ Yang},
       DOI = {10.2140/gt.2013.17.2431},
       URL = {https://doi.org/10.2140/gt.2013.17.2431},
}

@article {BM91,
    AUTHOR = {Bestvina, Mladen and Mess, Geoffrey},
     TITLE = {The boundary of negatively curved groups},
   JOURNAL = {J. Amer. Math. Soc.},
  FJOURNAL = {Journal of the American Mathematical Society},
    VOLUME = {4},
      YEAR = {1991},
    NUMBER = {3},
     PAGES = {469--481},
      ISSN = {0894-0347,1088-6834},
   MRCLASS = {20F32 (57M40)},
  MRNUMBER = {1096169},
MRREVIEWER = {Jerzy\ Dydak},
       DOI = {10.2307/2939264},
       URL = {https://doi.org/10.2307/2939264},
}

@article{DS21,
 author = {Danielski, Daniel and Kapovich, Michael and {\'S}wi{\k{a}}tkowski, Jacek},
 title = {Complete characterizations of hyperbolic {Coxeter} groups with {Sierpi{\'n}ski} curve boundary and with {Menger} curve boundary},
 fjournal = {Fundamenta Mathematicae},
 journal = {Fundam. Math.},
 issn = {0016-2736},
 volume = {267},
 number = {2},
 pages = {117--128},
 year = {2024},
 language = {English},
 doi = {10.4064/fm293-7-2024},
 zbMATH = {7960096},
 Zbl = {1554.20098}
}

@article {BO97,
    AUTHOR = {Bourdon, Marc},
     TITLE = {Sur la dimension de {H}ausdorff de l'ensemble limite d'une
              famille de sous-groupes convexes co-compacts},
   JOURNAL = {C. R. Acad. Sci. Paris S\'er. I Math.},
  FJOURNAL = {Comptes Rendus de l'Acad\'emie des Sciences. S\'erie I.
              Math\'ematique},
    VOLUME = {325},
      YEAR = {1997},
    NUMBER = {10},
     PAGES = {1097--1100},
      ISSN = {0764-4442},
   MRCLASS = {57S25},
  MRNUMBER = {1614024},
       DOI = {10.1016/S0764-4442(97)88712-5},
       URL = {https://doi.org/10.1016/S0764-4442(97)88712-5},
}

@article {PV05,
    AUTHOR = {Potyagailo, Leonid and Vinberg, Ernest},
     TITLE = {On right-angled reflection groups in hyperbolic spaces},
   JOURNAL = {Comment. Math. Helv.},
  FJOURNAL = {Commentarii Mathematici Helvetici. A Journal of the Swiss
              Mathematical Society},
    VOLUME = {80},
      YEAR = {2005},
    NUMBER = {1},
     PAGES = {63--73},
      ISSN = {0010-2571,1420-8946},
   MRCLASS = {20F55 (51F15 57M07 57M50)},
  MRNUMBER = {2130566},
MRREVIEWER = {Ruth\ Kellerhals},
       DOI = {10.4171/CMH/4},
       URL = {https://doi.org/10.4171/CMH/4},
}

@book {BO14,
     TITLE = {Thin groups and superstrong approximation},
    SERIES = {Mathematical Sciences Research Institute Publications},
    VOLUME = {61},
    EDITOR = {Breuillard, Emmanuel and Oh, Hee},
      NOTE = {Selected expanded papers from the workshop held in Berkeley,
              CA, February 6--10, 2012},
 PUBLISHER = {Cambridge University Press, Cambridge},
      YEAR = {2014},
     PAGES = {xii+362},
      ISBN = {978-1-107-03685-7},
   MRCLASS = {20-06},
  MRNUMBER = {3235652},
}

@incollection {DA01,
    AUTHOR = {Davis, Michael W.},
     TITLE = {Exotic aspherical manifolds},
 BOOKTITLE = {Topology of high-dimensional manifolds, {N}o. 1, 2 ({T}rieste,
              2001)},
    SERIES = {ICTP Lect. Notes},
    VOLUME = {9},
     PAGES = {371--404},
 PUBLISHER = {Abdus Salam Int. Cent. Theoret. Phys., Trieste},
      YEAR = {2002},
      ISBN = {92-95003-12-8},
   MRCLASS = {57P99 (57Q99 57R55)},
  MRNUMBER = {1937019},
}

@book {DA08,
    AUTHOR = {Davis, Michael W.},
     TITLE = {The geometry and topology of {C}oxeter groups},
    SERIES = {London Mathematical Society Monographs Series},
    VOLUME = {32},
 PUBLISHER = {Princeton University Press, Princeton, NJ},
      YEAR = {2008},
     PAGES = {xvi+584},
      ISBN = {978-0-691-13138-2; 0-691-13138-4},
   MRCLASS = {20F55 (05B45 05C25 51-02 57M07)},
  MRNUMBER = {2360474},
MRREVIEWER = {Ralf\ Koehl},
}

@article {DGK18,
    AUTHOR = {Danciger, Jeffrey and Gu\'eritaud, Fran\c cois and Kassel,
              Fanny},
     TITLE = {Convex cocompactness in pseudo-{R}iemannian hyperbolic spaces},
   JOURNAL = {Geom. Dedicata},
  FJOURNAL = {Geometriae Dedicata},
    VOLUME = {192},
      YEAR = {2018},
     PAGES = {87--126},
      ISSN = {0046-5755,1572-9168},
   MRCLASS = {57S30 (20F55 22E40 52A20 53C50)},
  MRNUMBER = {3749424},
MRREVIEWER = {Alejandro\ Ucan-Puc},
       DOI = {10.1007/s10711-017-0294-1},
       URL = {https://doi.org/10.1007/s10711-017-0294-1},
}

@article {BD05,
    AUTHOR = {Bagchi, Bhaskar and Datta, Basudeb},
     TITLE = {Combinatorial triangulations of homology spheres},
   JOURNAL = {Discrete Math.},
  FJOURNAL = {Discrete Mathematics},
    VOLUME = {305},
      YEAR = {2005},
    NUMBER = {1-3},
     PAGES = {1--17},
      ISSN = {0012-365X,1872-681X},
   MRCLASS = {57Q15 (05E25)},
  MRNUMBER = {2186679},
MRREVIEWER = {Alberto\ Cavicchioli},
       DOI = {10.1016/j.disc.2005.06.026},
       URL = {https://doi.org/10.1016/j.disc.2005.06.026},
}

@ARTICLE{DR05,
       author = {{Dranishnikov}, A.~N.},
        title = "{Cohomological dimension theory of compact metric spaces}",
      journal = {arXiv Mathematics e-prints},
         year = 2005,
        month = jan,
          doi = {10.48550/arXiv.math/0501523},
archivePrefix = {arXiv},
       eprint = {math/0501523},
 primaryClass = {math.GN},
       adsurl = {https://ui.adsabs.harvard.edu/abs/2005math......1523D}
}

@article {DR99,
    AUTHOR = {Dranishnikov, A. N.},
     TITLE = {Boundaries of {C}oxeter groups and simplicial complexes with
              given links},
   JOURNAL = {J. Pure Appl. Algebra},
  FJOURNAL = {Journal of Pure and Applied Algebra},
    VOLUME = {137},
      YEAR = {1999},
    NUMBER = {2},
     PAGES = {139--151},
      ISSN = {0022-4049,1873-1376},
   MRCLASS = {20H15 (20F65 57M07)},
  MRNUMBER = {1684267},
MRREVIEWER = {Nadia\ Benakli},
       DOI = {10.1016/S0022-4049(97)00202-8},
       URL = {https://doi.org/10.1016/S0022-4049(97)00202-8},
}

@article {BA82,
    AUTHOR = {Barnette, David},
     TITLE = {Generating the triangulations of the projective plane},
   JOURNAL = {J. Combin. Theory Ser. B},
  FJOURNAL = {Journal of Combinatorial Theory. Series B},
    VOLUME = {33},
      YEAR = {1982},
    NUMBER = {3},
     PAGES = {222--230},
      ISSN = {0095-8956,1096-0902},
   MRCLASS = {57Q15 (05B45 05C99 51M20 52A99)},
  MRNUMBER = {693361},
MRREVIEWER = {R.\ Blind},
       DOI = {10.1016/0095-8956(82)90041-7},
       URL = {https://doi-org.proxy.binghamton.edu/10.1016/0095-8956(82)90041-7},
}

@article {DO07,
    AUTHOR = {Dymara, Jan and Osajda, Damian},
     TITLE = {Boundaries of right-angled hyperbolic buildings},
   JOURNAL = {Fund. Math.},
  FJOURNAL = {Fundamenta Mathematicae},
    VOLUME = {197},
      YEAR = {2007},
     PAGES = {123--165},
      ISSN = {0016-2736,1730-6329},
   MRCLASS = {20E42 (20F67 57M07)},
  MRNUMBER = {2365885},
MRREVIEWER = {Anne\ Thomas},
       DOI = {10.4064/fm197-0-6},
       URL = {https://doi.org/10.4064/fm197-0-6},
}

@article {FI03,
    AUTHOR = {Fischer, Hanspeter},
     TITLE = {Boundaries of right-angled {C}oxeter groups with manifold
              nerves},
   JOURNAL = {Topology},
  FJOURNAL = {Topology. An International Journal of Mathematics},
    VOLUME = {42},
      YEAR = {2003},
    NUMBER = {2},
     PAGES = {423--446},
      ISSN = {0040-9383},
   MRCLASS = {57S30 (20F55 20F65)},
  MRNUMBER = {1941443},
MRREVIEWER = {Andrei\ Yu.\ Vesnin},
       DOI = {10.1016/S0040-9383(02)00014-9},
       URL = {https://doi.org/10.1016/S0040-9383(02)00014-9},
}

@incollection {G87,
    AUTHOR = {Gromov, M.},
     TITLE = {Hyperbolic groups},
 BOOKTITLE = {Essays in group theory},
    SERIES = {Math. Sci. Res. Inst. Publ.},
    VOLUME = {8},
     PAGES = {75--263},
 PUBLISHER = {Springer, New York},
      YEAR = {1987},
   MRCLASS = {20F32 (20F06 20F10 22E40 53C20 57R75 58F17)},
  MRNUMBER = {919829},
MRREVIEWER = {Christopher W. Stark},
       DOI = {10.1007/978-1-4613-9586-7_3},
       URL = {https://doi.org/10.1007/978-1-4613-9586-7_3},
}

@article{JA91,
author = {Jakobsche, W.},
journal = {Fundamenta Mathematicae},
language = {eng},
number = {2},
pages = {81-95},
title = {Homogeneous cohomology manifolds which are inverse limits},
url = {http://eudml.org/doc/211855},
volume = {137},
year = {1991},
}

@book {KAP09,
    AUTHOR = {Kapovich, Michael},
     TITLE = {Hyperbolic manifolds and discrete groups},
    SERIES = {Modern Birkh\"auser Classics},
      NOTE = {Reprint of the 2001 edition},
 PUBLISHER = {Birkh\"auser Boston, Boston, MA},
      YEAR = {2009},
     PAGES = {xxviii+467},
      ISBN = {978-0-8176-4912-8},
   MRCLASS = {57M50 (20F65 20H10 30F40 30F45 30F60 32G15)},
  MRNUMBER = {2553578},
       DOI = {10.1007/978-0-8176-4913-5},
       URL = {https://doi.org/10.1007/978-0-8176-4913-5},
}

@article {KK00,
    AUTHOR = {Kapovich, Michael and Kleiner, Bruce},
     TITLE = {Hyperbolic groups with low-dimensional boundary},
   JOURNAL = {Ann. Sci. \'Ecole Norm. Sup. (4)},
  FJOURNAL = {Annales Scientifiques de l'\'Ecole Normale Sup\'erieure.
              Quatri\`eme S\'erie},
    VOLUME = {33},
      YEAR = {2000},
    NUMBER = {5},
     PAGES = {647--669},
      ISSN = {0012-9593},
   MRCLASS = {20F67 (57M07)},
  MRNUMBER = {1834498},
MRREVIEWER = {Thomas\ Delzant},
       DOI = {10.1016/S0012-9593(00)01049-1},
       URL = {https://doi.org/10.1016/S0012-9593(00)01049-1},
}

@article {PS09,
    AUTHOR = {Przytycki, Piotr and \'{S}wi\c{a}tkowski, Jacek},
     TITLE = {Flag-no-square triangulations and {G}romov boundaries in
              dimension 3},
   JOURNAL = {Groups Geom. Dyn.},
  FJOURNAL = {Groups, Geometry, and Dynamics},
    VOLUME = {3},
      YEAR = {2009},
    NUMBER = {3},
     PAGES = {453--468},
      ISSN = {1661-7207,1661-7215},
   MRCLASS = {20F67 (20F55 20F65 57Q15)},
  MRNUMBER = {2516175},
MRREVIEWER = {Michel\ Coornaert},
       DOI = {10.4171/GGD/66},
       URL = {https://doi.org/10.4171/GGD/66},
}

@incollection {MI75,
    AUTHOR = {Milnor, John},
     TITLE = {On the {$3$}-dimensional {B}rieskorn manifolds {$M(p,q,r)$}},
 BOOKTITLE = {Knots, groups, and {$3$}-manifolds ({P}apers dedicated to the
              memory of {R}. {H}. {F}ox)},
    SERIES = {Ann. of Math. Stud.},
    VOLUME = {No. 84},
     PAGES = {175--225},
 PUBLISHER = {Princeton Univ. Press, Princeton, NJ},
      YEAR = {1975},
   MRCLASS = {57D70 (14B05 32C40)},
  MRNUMBER = {418127},
MRREVIEWER = {W.\ D.\ Neumann},
}

@article {SW20sp,
    AUTHOR = {\'{S}wi\k{a}tkowski, Jacek},
     TITLE = {Trees of metric compacta and trees of manifolds},
   JOURNAL = {Geom. Topol.},
  FJOURNAL = {Geometry \& Topology},
    VOLUME = {24},
      YEAR = {2020},
    NUMBER = {2},
     PAGES = {533--592},
      ISSN = {1465-3060,1364-0380},
   MRCLASS = {20F65 (54D80 57M07)},
  MRNUMBER = {4153649},
MRREVIEWER = {Jens\ Harlander},
       DOI = {10.2140/gt.2020.24.533},
       URL = {https://doi.org/10.2140/gt.2020.24.533},
}

@article {SW20,
    AUTHOR = {\'{S}wi\k{a}tkowski, Jacek},
     TITLE = {Trees of manifolds as boundaries of spaces and groups},
   JOURNAL = {Geom. Topol.},
  FJOURNAL = {Geometry \& Topology},
    VOLUME = {24},
      YEAR = {2020},
    NUMBER = {2},
     PAGES = {593--622},
      ISSN = {1465-3060,1364-0380},
   MRCLASS = {20F67 (20F65 57M07)},
  MRNUMBER = {4153650},
MRREVIEWER = {Igor\ Belegradek},
       DOI = {10.2140/gt.2020.24.593},
       URL = {https://doi.org/10.2140/gt.2020.24.593},
}

@book{moussong,
	author = "Moussong, Gabor",
	mrclass = "Thesis",
	mrnumber = "2636665",
	note = "Thesis (Ph.D.)--The Ohio State University",
	pages = "55",
	publisher = "ProQuest LLC, Ann Arbor, MI",
	title = "{Hyperbolic {C}oxeter groups}",
	url = "http://gateway.proquest.com/openurl?url_ver=Z39.88-2004&rft_val_fmt=info:ofi/fmt:kev:mtx:dissertation&res_dat=xri:pqdiss&rft_dat=xri:pqdiss:8824577",
	year = "1988"
}

@Article{benoist_harpe,
  author        = {Benoist, Yves and de la Harpe, Pierre},
  journal       = {Compos. Math.},
  title         = {{Adh{\'e}rence de {Z}ariski des groupes de {C}oxeter}},
  year          = {2004},
  issn          = {0010-437X},
  number        = {5},
  pages         = {1357--1366},
  volume        = {140},
  doi           = {10.1112/S0010437X04000338},
  fjournal      = {Compositio Mathematica},
  mrclass       = {20F55},
  mrnumber      = {MR2081159 (2005g:20059)},
  mrreviewer    = {O. V. Shvartsman},
  url           = {http://dx.doi.org/10.1112/S0010437X04000338},
}

@book {DA24,
    AUTHOR = {Davis, Michael W.},
     TITLE = {Infinite group actions on polyhedra},
    SERIES = {Ergebnisse der Mathematik und ihrer Grenzgebiete. 3. Folge. A
              Series of Modern Surveys in Mathematics [Results in
              Mathematics and Related Areas. 3rd Series. A Series of Modern
              Surveys in Mathematics]},
    VOLUME = {77},
 PUBLISHER = {Springer, Cham},
      YEAR = {[2024] \copyright 2024},
     PAGES = {xi+271},
      ISBN = {978-3-031-48442-1; 978-3-031-48443-8},
   MRCLASS = {20F65 (20F36 20F55 20F67 52B15 57K32 57M05 57Q15)},
  MRNUMBER = {4769472},
MRREVIEWER = {Michael\ Dougherty},
       DOI = {10.1007/978-3-031-48443-8},
       URL = {https://doi.org/10.1007/978-3-031-48443-8},
}

@book {Hauptbook,
    AUTHOR = {Ranicki, A. A. and Casson, A. J. and Sullivan, D. P. and
              Armstrong, M. A. and Rourke, C. P. and Cooke, G. E.},
     TITLE = {The {H}auptvermutung book},
    SERIES = {$K$-Monographs in Mathematics},
    VOLUME = {1},
      NOTE = {A collection of papers of the topology of manifolds},
 PUBLISHER = {Kluwer Academic Publishers, Dordrecht},
      YEAR = {1996},
     PAGES = {vi+190},
      ISBN = {0-7923-4174-0},
   MRCLASS = {57Q25 (57-06)},
  MRNUMBER = {1434100},
MRREVIEWER = {Oliver\ Attie},
       DOI = {10.1007/978-94-017-3343-4},
       URL = {https://doi-org.proxy.binghamton.edu/10.1007/978-94-017-3343-4},
}

@misc{computations,
  author = {Douba, Sami and Lee, Gye-Seon and Marquis, Ludovic and Ruffoni, Lorenzo},
  title = {SageMath files for the computation in this paper},
  year = {},
  publisher = {},
  journal = {},
  howpublished = {Available at \url{https://marquis.pages.math.cnrs.fr/ludo-marquis/sagemath/SageMath.html}},
  commit = {}
}

\end{document}